\documentclass[final,leqno,onefignum,onetabnum]{siamltex}
\usepackage{amsmath}
\usepackage{amsfonts}
\usepackage{mathrsfs}
\usepackage{amssymb}
\usepackage{verbatim}
\usepackage{url}
\usepackage{enumerate}
\usepackage[footnotesize,FIGBOTCAP,TABTOPCAP]{subfigure}
\usepackage[footnotesize]{caption}   
\usepackage{graphicx,color}
\usepackage{alltt}
\usepackage[margin=1.4in]{geometry}

\newcommand{\p}{M}

\newcommand{\B}{\mathcal{B}}
\newcommand{\R}{\mathcal{R}}
\usepackage{calc}  
\usepackage{amsfonts}
\usepackage[colorlinks=true, bookmarksopen,
          pdfcreator={pdftex},
            pdfsubject={algorithms},
            linkcolor={blue},
            anchorcolor={black},
            citecolor={red},
            filecolor={magenta},
            menucolor={black},
            pagecolor={red},
            plainpages=false,pdfpagelabels,
            urlcolor={blue}]{hyperref}
\let\oldeq\equation{}\def\equation{\par\vspace{-\parskip}\oldeq}
\newcounter{saveeqn}%

\newenvironment{example}[1][Example 1.]{\begin{trivlist}
\item[\hskip \labelsep {\bfseries #1}]}{\end{trivlist}}
\newenvironment{example2}[1][Example 2.]{\begin{trivlist}
\item[\hskip \labelsep {\bfseries #1}]}{\end{trivlist}}
\newenvironment{example3}[1][Example 3.]{\begin{trivlist}
\item[\hskip \labelsep {\bfseries #1}]}{\end{trivlist}}
\newtheorem{rems}{Remark}

\author{
 Vandana Sharma\footnotemark[2] \and Jeff Morgan\footnotemark[1]}

\begin{document}
\title{Global Existence of solutions to reaction diffusion systems with mass transport type boundary conditions}
\maketitle

\renewcommand{\thefootnote}{\fnsymbol{footnote}}

\footnotetext[2]{Department of Mathematical and Statistical Sciences, Arizona State University, Tempe, USA, AZ 85281. Email: \textup{\nocorr \texttt{vandanas@asu.edu}}.}
\footnotetext[1] {Department of Mathematics, The University of Houston, Houston, USA,  TX 77004. Email: \textup{\nocorr \texttt{jjmorgan@central.uh.edu}}. 
The authors acknowledge the generous support of NSF grant  DMS-0714864. }  
      
\begin{abstract}
We consider a reaction-diffusion system where some components
react and diffuse on the boundary of a region, while other components diffuse in
the interior and react with those on the boundary through mass transport. We establish  local well-posedness and global existence of solutions for these systems using classical
potential theory and linear estimates for initial boundary value problems.
\end{abstract}






\section{Introduction}\label{sec:intro}
 The idea that reaction-diffusion phenomena is essential to the growth of living organisms seems quite intuitive. Indeed, it would be rather hard to envision how any organism could grow and operate without moving its constituents around and using them in various bio-chemical reactions  \cite{RefWorks:146}. For example, bacterial cytokinesis is one process which can be modeled by reaction-diffusion systems.  During the bacterial cytokinesis process, a proteinaceous contractile ring assembles in the middle of the cell. The  ring tethers to the membrane and contracts to form daughter cells; that is, the ``cell divides".  One mechanism that centers the ring involves the pole-to-pole oscillation of proteins Min C, Min D and Min E. Oscillations cause the average concentration of Min C, an inhibitor of the ring assembly, to be lowest at the midcell and highest near the poles \cite{RefWorks:142}, \cite{RefWorks:143}. This centering mechanism, relating molecular-level interactions to supra-molecular ring positioning can be modelled as a system of semilinear parabolic equations.  The multi-dimensional version of the evolution  of the Min concentrations can be described as a special case of the reaction-diffusion system
\begin{align}\label{sy15}
 u_t\nonumber&=  D \Delta u+H(u)
 & x\in \Omega, \quad&0<t<T
\\\nonumber v_t&=\tilde D \Delta_{M} v+F(u,v)& x\in M,\quad& 0<t<T\\ D\frac{\partial u}{\partial \eta}&=G(u,v) & x\in M, \quad&0<t<T\\\nonumber
u&=u_0  &x\in\Omega ,\quad& t=0\\\nonumber v&=v_0 & x\in M ,\quad &t=0\end{align}
where $\Omega$ is a bounded domain in $\mathbb{R}^n$, $n\geq 2$, with smooth boundary M, $\Delta$ and $\Delta_M$ denote the Laplace and Laplace Beltrami operators, $\eta$ is the unit outward normal vector to $\Omega$ at points on $M$, and $D$ and $\tilde D$ are $k\times k$ and $m\times m$ diagonal matrices with positive diagonal entries $\lbrace d_j\rbrace_{1\leq j\leq k}$ and $\lbrace\tilde d_i\rbrace_{1\leq i\leq m}$ respectively. $F:\mathbb{R}^k\times \mathbb{R}^m\rightarrow \mathbb{R}^m$,  $G:\mathbb{R}^k\times \mathbb{R}^m\rightarrow \mathbb{R}^k$, $H:\mathbb{R}^k\rightarrow \mathbb{R}^k$, and $u_0 $ and $v_0$ are componentwise nonnegative smooth functions that satisfy the compatibility condition\[ D{\frac{ \partial {u_0}}{\partial \eta}} =G(u_0,v_0)\quad \text{on $M.$}\]
 For this model, $\Omega$ may represent the cell cytoplasm and $M$ may represent its membrane. There are some components that are bound to the membrane, and other components that move freely in the cytoplasm. Also, the components on the membrane and cytoplasm react together on the membrane through mass action and boundary transport. In Section 7, we present two applications associated with ($\ref{sy15}$), with one modeling the chemical reaction involving Min protiens for positioning of the ring, explained in \cite{RefWorks:143}. We point out the study in \cite{RefWorks:142} that also modeled these reactions.

In general, system ($\ref{sy15}$) is somewhat reminiscent of two component systems where both of the unknowns react and diffuse inside $\Omega$, with various homogeneous boundary conditions and nonnegative initial data. In that setting, global well-possedness and uniform boundedness has been studied by many researchers, and we refer the interested reader to the excellent survey of Pierre \cite{ RefWorks:86}.

In the remainder of the introduction, we assume $H=0$ and $k=m=1$. A fundamental mathematical question concerning global existence for ($\ref{sy15}$) asks, what conditions on $F$ and $G$ will guarantee that ($\ref{sy15}$) has global solutions, and how are these conditions related to the results listed in \cite{ RefWorks:86}?  The focus of this paper is to give a partial answer to this question and to apply our results to ($\ref{sy15}$).

From a physical standpoint, it is natural to ask under what conditions the solutions of $(\ref{sy15})$ are nonnegative, and the total mass is either conserved or reduced. It is also important to ask whether these conditions arise in problems similar to the above mentioned cell biology system. Conditions that are similar in spirit to those given in \cite{ RefWorks:110}, \cite{ RefWorks:85} and \cite{ RefWorks:86}  result in nonnegative solutions for system ($\ref{sy15}$). More precisely, ($\ref{sy15}$) has nonnegative solutions for all choices of nonnegative initial data $u_0$  and $v_0$ if and only if $F$, $G$,  and $H$ are quasi-positive. That is $F(a,0), G(0,a)\geq 0$ whenever $a\geq 0$ (recall from above that $H=0$ in the remainder of this introduction). Also, some control of total mass can be achieved by assuming there exists $\alpha>0$ such that \begin{align}\label{mass}
F(u,v)+ G(u,v)&\leq \alpha(u+v+1)\quad\text{ for all }  u, v \geq 0.
\end{align}
Assumption $(\ref{mass})$ (discussed later),  generalizes mass conservation by implying that total mass, $\int_{\Omega} u(x,t) \ dx + \int_{M} v(\zeta,t) \ d\sigma $, grows at most exponentially in time $t$.

 We suspect that the natural conditions, quasipositivity and conservation of mass, are not sufficient to obtain global existence in  ($\ref{sy15}$), and that it is possible to construct an example along the  same lines as constructed in \cite{ RefWorks:124}. To this end, we  impose a condition similar to Morgan's intermediate sums \cite{ RefWorks:120} and  \cite{ RefWorks:123}. Namely,
there exists a constant $K_g>0$ such that \[\quad G(\zeta,\nu)\leq K_g(\zeta+\nu+1)\quad\text{for all}\quad\nu, \zeta \geq 0.\]
In addition, we adopt  a natural assumption of polynomial growth, which has been considered in the context of chemical and biological modeling (see Horn and Jackson \cite{ RefWorks:131}). That is, there exists $ l\in \mathbb{N}$ and $K_f>0$ such that \[\quad F(u,v)\leq K_f( u+ v+1)^l\quad\text{for all}\quad v\geq 0,\  u\geq 0.\] In our analysis, we extend recent results of Huisken and  Polden \cite{ RefWorks:15}, Polden \cite{ RefWorks:16}, and Sharples  \cite{ RefWorks:114} associated with $W^{2,1}_2(M\times(0,T))$ results for solutions to linear Cauchy problems on a membrane. We also verify and make use of a remark of Brown \cite{ RefWorks:81} which states that  if $d>0$ and the Neumann data $\gamma$  lies in $L_{p}(M\times(0,T))$ for $p>n+1$, then the solution to \begin{align}\label{sy3}
 \varphi_t\nonumber&= d\Delta \varphi &
 x\in \Omega,\quad &0<t< T
\\ d\frac{\partial \varphi}{\partial \eta}&=\gamma &x\in M,\quad &0<t< T\\\nonumber
\varphi&= 0 & x\in\Omega ,\quad& t=0
\end{align}is  H\"{o}lder continuous on  $\overline{\Omega}\times(0,  T)$. We provide the proof of this result in section 5 for completeness of our arguments.

Note that the results of Amann \cite{ RefWorks:6} can be used to guarantee the local well posedness of  ($\ref{sy15}$) subject to appropriate conditions on initial data and on the functions $F$ and $G$. However, those results do not  provide the explicit estimates that are needed in our setting. Our approach keeps the analysis on comparatively simpler $L_p$ spaces.

 It is worth mentioning that some of the results in section 5 are valid for domains that are only $C^{1}$. Handling cases with weak smoothness conditions on curves or domain boundaries was one of the motivations for the results obtained in $\cite{RefWorks:81}, \cite{RefWorks:107}, \cite{RefWorks:106}$ and $\cite{RefWorks:105}$ , and these results may be of independent interest.


\section{Notations, Definitions and Preliminary Estimates}
\setcounter{equation}{0}
Throughout this paper, $n\geq 2$ and $\Omega$ is a bounded domain in $\mathbb{R}^n$ with smooth boundary M ($\partial \Omega$) belonging to the class $C^{2+\mu}$ with $\mu>0$ such that $\Omega$ lies locally on one side of its  boundary. $\eta$ is the unit outward normal (from $\Omega$) to $M$, and $\Delta$ and $\Delta_{\p}$ are the Laplace and the Laplace Beltrami operators, respectively. For more details, see Rosenberg \cite{RefWorks:50} and Taylor \cite{RefWorks:62}. In addition, $m, k, n,i$ and $ j$ are positive integers, $D$ and $\tilde D$ are $k\times k$ and $m\times m$ diagonal matrices with positive diagonal entries $\lbrace d_j\rbrace_{1\leq j\leq k}$ and $\lbrace\tilde d_i\rbrace_{1\leq i\leq m}$, respectively. 
\subsection{Basic Function Spaces}
Let $\B$  be a bounded domain on $\mathbb{R}^m$ with smooth boundary such that $\B$ lies locally on one side of $\partial\B$. We define all function spaces on $\B$ and $\B_T=\B\times(0,T)$. 
$L_q(\B)$ is the Banach space consisting of all measurable functions on $\B$ that are $q^{th}(q\geq 1)$ power summable on $\B$. The norm is defined as\[ \Vert u\Vert_{q,\B}=\left(\int_{\B}| u(x)|^q dx\right)^{\frac{1}{q}}\]
Also, \[\Vert u\Vert_{\infty,\B}= ess \sup\lbrace |u(x)|:x\in\B\rbrace.\]
Measurability and summability are to be understood everywhere in the sense of Lebesgue.

If $p\geq 1$, then $W^2_p(\B)$ is the Sobolev space of functions $u:\B\rightarrow \mathbb{R}$ with generalized derivatives, $\partial_x^s u$ (in the sense of distributions) $|s|\leq 2$ belonging to $L_p(\B)$.  Here $s=(s_1,s_2,$...,$s_n),|s|=s_1+s_2+..+s_n$, $|s|\leq2$, and $\partial_x^{s}=\partial_1^{s_1}\partial_2^{s_2}$...$\partial_n^{s_n}$ where $\partial_i=\frac{\partial}{\partial x_i}$. The norm in this space is \[\Vert u\Vert_{p,\B}^{(2)}=\sum_{|s|=0}^{2}\Vert \partial_x^s u\Vert_{p,\B} \]

Similarly, $W^{2,1}_p(\B_T)$ is the Sobolev space of functions $u:\B_T\rightarrow \mathbb{R}$ with generalized derivatives, $\partial_x^s\partial_t^r u$ (in the sense of distributions) where $2r+|s|\leq 2$ and each derivative belonging to $L_p(\B_T)$. The norm in this space is \[\Vert u\Vert_{p,\B_T}^{(2)}=\sum_{2r+|s|=0}^{2}\Vert \partial_x^s\partial_t^r u\Vert_{p,\B_T} \] In addition to $W^{2,1}_p(\B_T)$, we will encounter other spaces with two different ratios of upper indices, 
$W_2^{1,0}(\B_T)$, $W_2^{1,1}(\B_T)$, $V_2(\B_T)$, $V_2^{1,0}(\B_T)$, and $V_2^{1,\frac{1}{2}}(\B_T)$ as defined in \cite{RefWorks:65}.

We also introduce $W^l_p(\B)$, where $l>0$ is not an integer, because initial data will be taken from these spaces. The space $W^l_p(\B)$ with nonintegral $l$, is a Banach space consisting of elements of $W^{[l]}_p$ ([$l$] is the largest integer less than $ l$) with the finite norm\[\Vert u\Vert_{p,\B}^{(l)} =\langle u\rangle_{p,\B}^{(l)}+\Vert u\Vert_{p,\B}^{([l])} \]
where  \[\Vert u\Vert_{p,\B}^{([l])}=\sum_{s=0}^{[l]}\Vert \partial_x^s u\Vert_{p,\B} \] and 
\[\langle u\rangle_{p,\B}^{(l)}=\sum_{s=[l]}\left(\int_\B dx\int_\B{|\partial_x^s u(x)-\partial_y^s u(y)|}^p.\frac{dy}{|x-y|^{n+p(l-[l])}}\right)^\frac{1}{p}\]
$W^{l,\frac{l}{2}}_p(\partial\B_T)$ spaces with non integral $l$ also play an important role in the study of boundary value problems with nonhomogeneous boundary conditions, especially in the proof of exact estimates of their solutions.  It is a Banach space when $p\geq 1$, which is defined by means of parametrization of the surface $\partial\B$. For a rigorous treatment of these spaces, we refer the reader to page 81 of Chapter 2 of \cite{RefWorks:65}.

The use of the spaces $W^{l,\frac{l}{2}}_p(\partial\B_T)$ is connected to the fact that the differential properties of the boundary values of functions from $W^{2,1}_p(\B_T)$ and of certain of its derivatives, $\partial_x^s\partial_t^r$, can be exactly described in terms of the spaces $W^{l,\frac{l}{2}}_p(\partial\B_T)$, where $l=2-2r-s-\frac{1}{p}$.

For $0<\alpha,\beta<1$,  $C^{\alpha,\beta}(\overline{\B_T})$ is the Banach space of H\"{o}lder continuous functions $u$ with the finite norm \[ |u|^{(\alpha)}_{\overline\B_T}= \sup_{(x,t)\in{\B_T}} |u(x,t)|+[u]^{(\alpha)}_{x,\B_T}+[ u]^{(\beta)}_{t,\B_T}\]
where \[ [u]^{(\alpha)}_{x, {\overline\B_T}}= \sup_{\substack{(x,t),(x',t)\in {\B_T}\\ x\ne x'}}\frac{|u(x,t)-u(x',t)|}{|x-x'|^{\alpha}}\] and \[  [u]^{(\beta)}_{t, {\overline\B_T}}= \sup_{\substack{(x,t),(x,t')\in {\B_T}\\ t\ne t'}}\frac{|u(x,t)-u(x,t')|}{|t-t'|^\beta}\]
We shall denote the space $C^{\frac{\alpha}{2},\frac{\alpha}{2}}(\overline\B_T)$ by $C^{\frac{\alpha}{2}}(\overline\B_T)$. $ C(\B_T,\mathbb{R}^n)$ is the set of all continuous functions $u: \B_T \rightarrow \mathbb{R}^n$, and
$ C^{1,0}(\B_T,\mathbb{R}^n)$ is the set of all continuous functions $u: \B_T\rightarrow \mathbb{R}^n$ for which $ u_{x_i}$ is continuous for all $1\leq i\leq n$.
$ C^{2,1}(\B_T,\mathbb{R}^n)$ is the set of all continuous functions $u: \B_T \rightarrow \mathbb{R}^n$ having continuous derivatives $u_{x_i},u_{{x_i}{x_j}}\ \text{and}\ u_t$ in $\B_T$.
Note that similar definitions can be given on $\overline\B_T$. Moreover notations and definitions for H\"{o}lder and Sobolev Spaces on manifolds  are similar to the ones used in the Handbook of Global analysis \cite{ RefWorks:71}.
More developments on Sobolev spaces, Sobolev inequalities, and the notion of best constants may be found in \cite{ RefWorks:133}, \cite{ RefWorks:135}, \cite{ RefWorks:111} and \cite{ RefWorks:62}. 
\subsection{Preliminary Estimates}
For completeness of  our arguments, we state the following results, which will help us obtain a priori estimates for the Cauchy problem on the manifold $M$, and prove the existence of solutions in $W^{2,1}_p(M_T)$. Lemmas $\ref{L1}$, $\ref{i}$  and $\ref{L3}$ can be found on page 341, Chapter 4  in \cite{RefWorks:65}, as $(2.24)$ and $(2.25)$ on page 49 in \cite{RefWorks:69}, and \cite{RefWorks:111} respectively. Lemma $\ref{Hol}$ is stated as Lemma 3.3 in Chapter 2 of \cite{RefWorks:65}.

Let $\mathcal{B}$ be a  bounded domain in $\mathbb{R}^m$ with smooth boundary $\partial\mathcal{B}$ belonging to the class $C^{2+\mu}$ with $\mu>0$ such that $\mathcal{B}$ lies locally on one side of the boundary $\partial\mathcal{B}$. Let $T>0$ and $p>1$. Suppose $\Theta\in L_{p}(\B_T)$, $w_0\in W_p^{2}(\B)$, $\gamma\in {W_{p}}^{2-\frac{1}{p},1-\frac{1}{2p}}(\partial\B_T)$. Also, let the coefficient matrix $(a_{i,j})$ be symmetric and continuous on $\overline{\B_T}$, and satisfy the uniform ellipticity condition. That is for some $ \lambda>0$ \[ \sum\limits_{i,j=1}^n a_{ij}(x,t)\xi_i\xi_j\geq \lambda |\xi|^2\  \text{for all}\  (x,t)\in \overline{\B_T} \ \text{and for all}\  \xi\in \mathbb{R}^n \]Finally, let the coefficients $a_i$ be continuous on $\overline {\B_T}$. Consider the problem
\begin{eqnarray}\label{m1}
  \frac{\partial w }{\partial t}  - \displaystyle\sum\limits_{i,j=1}^{n} a_{ij}(x,t)\frac{\partial^{2}w}{\partial x_{i}\partial x_{j}}+\displaystyle\sum\limits_{i=1}^{n} a_{i}(x,t)\frac{\partial w}{\partial x_{i}}\nonumber &=&\Theta(x,t) \hspace{.3in} (x,t)\in \B_T\quad\quad\quad\\
 w&=&\gamma(x,t) \hspace{.3in}(x,t)\in \partial\B_T\\ \nonumber
{ w\big|}_{t=0} &=& {w_0}(x) \hspace{.5in} x \in \B\nonumber
\end{eqnarray}
\begin{lemma}\label{L1}
Let p $> 1$ with $p\neq\frac{3}{2}$, and in the case $p>\frac{3}{2}$, assume the compatibility condition of zero order, $w_0|_{\partial\B}=\gamma|_{t=0}$. Then $(\ref{m1})$ has a unique solution $ w\in{W_{p}}^{2,1}{(\B_T)}$, and there exists $C>0$ depending on $T, p$ and $\B$, and independent of $\Theta, w_0$ and $\gamma$ such that
\[{\Vert w\Vert}_{p,\B_T}^{(2)}\leq C({\Vert \Theta\Vert}_{p,\B_T}+{\Vert w_0\Vert}_{p,\B}^{(2-\frac{2}{p})}+{\Vert \gamma\Vert}_{p, \partial\B_T}^{(2-\frac{1}{p},1-\frac{1}{2p})})\]\end{lemma}
\begin{lemma}\label{Hol}
Suppose $q\geq p$, $2-2r-s-\left(\frac{1}{p}-\frac{1}{q}\right)(m+2)\geq 0$ and $0<\delta\leq \min\lbrace d;\sqrt{T}\rbrace$. Then there exists $c_1, c_2>0$ depending on $r, s, m, p$ and $\B$ such that 
\[ \Vert D_t^r D_x^s u\Vert _{q,\B_T}\leq c_1 \delta^{2-2r-s-\left(\frac{1}{p}-\frac{1}{q}\right)(m+2)}\Vert u\Vert^{(2)}_{p,\B_T}+ c_2 \delta ^{-(2r+s+\left(\frac{1}{p}-\frac{1}{q}\right)(m+2))}\Vert u\Vert_{p,\B_T}\] for all $u\in W_p^{2,1}(\mathcal{B_T})$. Moreover, if $2-2r-s-\frac{(m+2)}{p}>0$, then for $0\leq \alpha<2-2r-s-\frac{(m+2)}{p}$ there exist constants $c_3, c_4$ depending on $r, s, m, p$ and $\B$ such that \[ |D_t^r D_x^s u|^{(\alpha)}_{\B_T}\leq c_3 \delta ^{2-2r-s-\frac{m+2}{p}-\alpha}\Vert u\Vert^{(2)}_{p,\B_T}+ c_4 \delta^{-(2r+s+\frac{(m+2)}{p}+\alpha)}\Vert u\Vert_{p,\B_T}\] for all $u \in W_p^{2,1}(\mathcal{B_T})$.
\end{lemma}\\
\begin{corollary}
Suppose the conditions of Lemma $\ref {L1}$ are fulfilled and  $p>\frac{m+2}{2}$. Then there exists $\hat c>0$ depending on $m, p$ and $\B$ such that the solution of problem $(\ref{m1})$ is H\"{o}lder continuous, and \[|w|^{(2-\frac{m+2}{p})}_{\B_T}\leq \hat c \Vert w\Vert^{(2)}_{p,\B_T}\] 
\end{corollary}\\
\begin{lemma}\label{i}
Suppose $1<p<\infty$. If $p< m$ then $ W^1_p(\B)$ embedds continuously into $W_p^{(1-\frac{1}{p} )}(\partial\B)$ and $L_q(\B)$ for $p\leq q\leq p^*=\frac{mp}{m-p}$. Furthermore, if $\epsilon>0$ there exists $C_{\epsilon}>0$ such that  
 \[{\Vert v\Vert}^p_{q,\B}\nonumber \leq \epsilon {\Vert  v_{x}\Vert}^p_{p,\B }+C_{\epsilon}{\Vert  v\Vert}^p_{1,\B }\] for all $v\in W^1_p(\B)$, and 
\[\Vert v\Vert^2_{2,\partial\B}\leq \epsilon {\Vert  v_{x}\Vert}^2_{2,\B }+C_{\epsilon}{\Vert  v\Vert}^2_{2,\B }\] for all $v\in W^1_2(\B)$.
\end{lemma}\\
\begin{lemma}\label{L3.5}
Let $p>m$ and $0<\alpha<1-\frac{m}{p}$. Then $W^{1}_p(\B)$ embedds compactly in $ C^{\alpha}(\overline\B)$. 
\end{lemma}\\
\begin{lemma}\label{L3}
Let $M$ be a compact Riemannian manifold of dimension $m\geq 1$ and $p>m$. Then the embedding $W^{1}_p(M)\subset C^{\alpha}(M)$ is compact for all  $0<\alpha<1-\frac{m}{p}$.
\end{lemma}\\

The following result follows from the Gagliardo Nirenberg inequality in $\cite{RefWorks:52}$ on bounded $C^{1}$ domains, and Young's inequality on page 40 in \cite{RefWorks:69}.\\
\begin{lemma}\label{L4}
Let $\epsilon>0$ and $1<p<\infty$. Then there exists  $ C_{\epsilon, p}>0$ such that \[{\Vert v_{x}\Vert}_{p,\B }\le \epsilon {\Vert v_{xx}\Vert}_{p,\B }+C_{\epsilon,p} {\Vert v\Vert}_{p,\B}\] for all $v\in$ $W^{2}_p(\B)$.
\end{lemma}

\section{Statements of Main Results}

\setcounter{equation}{0}
 The primary concern of this work is the system
\begin{align}\label{sy5}
 u_t\nonumber&=  D\Delta u+H(u)
 & x\in \Omega, \quad&0<t<T
\\\nonumber v_t&=\tilde D\Delta_{M} v+F(u,v)& x\in M,\quad& 0<t<T\\ D\frac{\partial u}{\partial \eta}&=G(u,v) & x\in M, \quad&0<t<T\\\nonumber
u&=u_0  &x\in\Omega ,\quad& t=0\\\nonumber v&=v_0 & x\in M ,\quad &t=0\end{align}
where 
 $D$ and $\tilde D$ are $k\times k$ and $m\times m$ diagonal matrices with positive diagonal entries, $F=(F_i):\mathbb{R}^k\times\mathbb{R}^m\rightarrow \mathbb{R}^m, G=(G_j):\mathbb{R}^k\times\mathbb{R}^m\rightarrow \mathbb{R}^k$ and $H=(H_j):\mathbb{R}^k \rightarrow \mathbb{R}^k$, and
$ u_0=( {u_0}_j)\in W_p^{2}(\Omega)$, $v_0= ({v_0}_i)\in W_p^{2}(M)$ with $p>n$. Also, $u_0 $  and $v_0$ satisfy the compatibility condition\[ D{\frac{ \partial {u_0}}{\partial \eta}} =G(u_0,v_0)\quad \text{on $M.$}\]
\begin{rems}
Since $p>n$, $u_0$ and $v_0$ are H\"{o}lder continuous functions on $\overline \Omega$  and $ M$ respectively (see $\cite{RefWorks:76}$, $\cite{RefWorks:52})$. 
\end{rems}\\
\begin{definition}\label{blah}
A function $(u,v)$ is said to be a $\it solution$ of $\left (\ref{sy5}\right)$ if and only if \[u \in C(\overline \Omega\times[0,T),\mathbb{R}^k)\cap C^{1,0}(\overline \Omega\times(0,T),\mathbb{R}^k)\cap C^{2,1}( \Omega\times(0,T),\mathbb{R}^k)\] and \[v \in C(M\times[0,T),\mathbb{R}^m)\cap C^{2,1}( M\times(0,T),\mathbb{R}^m) \] such that $(u,v)$ satisfies $\left (\ref{sy5}\right)$. If $T=\infty$, the solution is said to be a {\it global solution.}
\end{definition}
Moreover, a solution $(u,v)$ defined for $0\leq t<b$ is a $\it  maximal\  solution$ of $\left (\ref{sy5}\right)$ if and only if $(u,v)$ solves $\left (\ref{sy5}\right)$ with $T=b$, and if $d>b$ and $(\tilde u,\tilde v)$ solves $\left (\ref{sy5}\right)$ for $T=d$ then there exists $0<c<b$ such that $(u(\cdot,c),v(\cdot,c))\ne(\tilde u(\cdot,c), \tilde v(\cdot,c))$.

We say $F$, $G$ and $H$ are $\it quasi positive$ if and only if $F_i(\zeta,\xi)\geq 0$ whenever $\xi\in\mathbb{R}_+^m$ and $\zeta\in \mathbb{R}_+^k$ with $\xi_i=0$ for $i=1,...,m$, and $G_j( \zeta,\xi)\geq 0$, $H_j( \zeta)\geq 0$ whenever $\xi\in\mathbb{R}_+^m$ and $ \zeta\in\mathbb{R}_+^k$ with $\zeta_j=0$, for $j=1,...,k.$\\

The purpose of this study is to give sufficient conditions guaranteeing that $\left (\ref{sy5}\right)$ has a global solution. The following Theorems comprise  local and global existence of the solution.

\begin{theorem}\label{lo}
 Suppose $F$, $ G$ and $H$ are locally Lipschitz. Then there exists $T_{\max}>0$ such that $\left (\ref{sy5}\right)$ has a unique, maximal solution $(u,v)$ with $T=T_{\max}$.  Moreover, if $T_{\max}<\infty$ then 
\[\displaystyle \limsup_{t \to T^-_{\max}}\Vert u(\cdot,t)\Vert_{\infty,\Omega}+\displaystyle \limsup_{t \to T^-_{\max} }\Vert v(\cdot,t)\Vert_{\infty,M}=\infty\]
\end{theorem}
In addition to the assumptions stated above, we say condition $V_{i,j}$ holds for $1\leq j\leq k$ and $1\leq i\leq m$ if and only if 
\begin{itemize}
\item[($V_{i,j}1$)] There exist $\alpha,\beta,\sigma>0$ such that \[\sigma F_i(\zeta,\nu)+ G_j(\zeta,\nu)\leq \alpha(\zeta_j+\nu_i+1)\quad\text{and}\quad H_j(\zeta)\leq \beta(\zeta_j+1)\quad\text{ for all} \quad\nu \in\mathbb{R}^m_{\geq 0},\  \zeta \in\mathbb{R}^k_{\geq 0}\]
\item[($V_{i,j}2$)]There exists $K_g>0$ such that \[\quad G_j(\zeta,\nu)\leq K_g(\zeta_j+\nu_i+1)\quad\text{for all}\quad\nu \in\mathbb{R}^m_{\geq 0},\  \zeta \in\mathbb{R}^k_{\geq 0}\]
\item[($V_{i,j}3$)]There exists $l \in \mathbb{N}$ and $K_f>0$ such that \[\quad  F_i(\zeta,\nu)\leq K_f( |\zeta|+|\nu|+1)^l\quad\text{for all} \quad\nu \in\mathbb{R}^m_{\geq 0},\  \zeta \in\mathbb{R}^k_{\geq 0}\]
\end{itemize}
\begin{rems}
$(V_{i,j}2)$ is related to the so - called linear ``intermediate sums" condition used by Morgan in $\cite{RefWorks:120}$, $\cite{RefWorks:123}$ in the special case when the system has only two equations. This condition in $\cite{RefWorks:120}$, $\cite{RefWorks:123}$, as well as $\cite{RefWorks:86}$ pertains to interactions between the first m-1 equations in an m component system. Again, see $\cite{RefWorks:120}$, $\cite{RefWorks:123}$ and $\cite{RefWorks:86}$. $(V_{i,j}1)$ helps control mass, and allows higher order nonlinearities in $F$, but requires cancellation of high-order positive terms by G. $(V_{i,j}3)$ implies $F$ is polynomially bounded above.
\end{rems}\\
\begin{rems}
We will show that $(V_{i,j}1)$ provides $L_1$ estimates for $u_j$ on $\Omega$ and $M$, and $v_i$ on $M$. $(V_{i,j}2)$ helps us bootstrap $L_p$ estimates for  $u_j$ on $M\times(0,T_{max})$ and $\Omega\times(0,T_{max})$, and  $v_i$ on $ M\times(0,T_{max})$. Finally, $(V_{i,j}2)$ and $(V_{i,j}3)$ allow us to use $L_p$ estimates to obtain sup norm estimates on $u_j$ and $v_i$. 
\end{rems}\\

\begin{theorem}\label{great}
Suppose $F$, $G$ and $H$ are locally Lipschitz,  quasi positive, and $u_0, v_0$ are componentwise nonnegative functions. Also, assume that for each $1\leq j\leq k$ and $1\leq i\leq m$, there exists $l_i\in\lbrace 1,...,k\rbrace$ and $k_j\in\lbrace1,...,m\rbrace$ so that both $V_{i,l_i}$ and $V_{k_j,j}$ are satisfied. Then  $(\ref{sy5})$ has a unique component-wise nonegative global solution.
\end{theorem}\\

\begin{corollary}
Suppose $k=m=1,$ $ F,$ $G$ and $H$ are locally Lipschitz and quasipositive, and $u_0, v_0$ are nonnegative functions. If $V_{1,1}$ is satisfied, then  $(\ref{sy5})$ has a unique nonnegative global solution.
\end{corollary}\\

In the process of obtaining our results, we will derive $W^{2,1}_p(M_T)$ estimates of the Cauchy problem on $M_T$, and H\"{o}lder estimates of the solution to the Neumann problem on $\Omega_T$. The H\"{o}lder estimates for the solution to the Neumann problem are given as a comment in Brown \cite{ RefWorks:81}. We give the statement as Theorem $\ref{n}$ below, and supply a proof in section 5.  Let $\tilde d,  d>0$. Consider the systems
\begin{align}
 \Psi_t &=\tilde d \Delta_{\p} \Psi+f &(\xi,t)&\in M\times(0,T)\nonumber \\
 {\Psi\big|}_{t=0}&= \Psi_0 &\xi&\in M \label{sys2}
\end{align}
 and \begin{align}\label{m2}
 \varphi_t\nonumber&=  d\Delta \varphi+\theta &
 x\in \Omega,&\quad 0<t<T
\\ d\frac{\partial \varphi}{\partial \eta}&=\gamma & x\in M,&\quad 0<t< T\\\nonumber
\varphi&=\varphi_0 & x\in\Omega ,&\quad t=0
\end{align}
\begin{theorem}\label{3}
If $1<p<\infty$ and $T>0$, then there exists $\hat C_{p,T}>0$ such that whenever $\Psi_0\in W^{2-\frac{2}{p}}_p(M)$ and  $f\in L_p(M_T)$, there exists a unique solution $\Psi\in W_p^{2,1}(M_T)$ of $(\ref{sys2})$, and \[\Vert \Psi\Vert_{p,M_T}^{(2)}\leq\hat C_{p,T}(\Vert f\Vert_{p, M_T}+\Vert \Psi_0\Vert_{p,\p}^{(2-\frac{2}{p})})\]
\end{theorem}
\begin{theorem}\label{n}
Suppose $p>n+1$ and $T>0$ and $\theta\in L_p(\Omega\times(0, T))$, $\gamma\in L_p(M\times(0, T))$ and  $\varphi_0\in W^{2}_p(\Omega)$ such that  \[  d\frac{\partial {\varphi_0}}{\partial \eta} =\gamma(x,0)\quad{\text {on $M$.}}\] Then there exists $C_{p,T}>0$ independent of $\theta, \gamma$ and $\varphi_0$ and a unique weak solution $\varphi \in V_2^{1,\frac{1}{2}}(\Omega_T)$ of $(\ref{m2})$, such that if $0<\beta<1-\frac{n+1}{p}$ then \[ \vert \varphi\vert^{(\beta)}_{\Omega_{\hat T}}\leq C_{p,T}( \Vert \theta\Vert_{p,\Omega_{ T}}+\Vert \gamma\Vert_{p,M_{ T}}+\Vert \varphi_0\Vert^{(2)}_{p,\Omega})\]
\end{theorem}
The proofs of Theorems $\ref{3}$ and $\ref{n}$ are given in sections 4 and 5. The remaining results are proved in section 6, and examples are given in section 7.

\section{$W^{2,1}_p$ estimates for the Cauchy problem on a manifold}
\setcounter{equation}{0}
Let $n\geq 2$ and $M$ be a compact $n-1$ dimensional Riemannian manifold without boundary. Consider $(\ref{sys2})$
where $\tilde d>0$, $f\in L_p(M_T)$ and $\Psi_0\in W_p^{2-\frac{2}{p}}(M)$. Searching the literature, we surprisingly could not find $W^{2,1}_p(M_T)$ estimates for the solutions to $(\ref{sys2})$. Tracing through the work in this direction, we found that Huisken and Polden \cite{RefWorks:16} and \cite{RefWorks:15}, and J.J Sharples \cite{RefWorks:114} give a result in the setting where $p=2$.  
 Using their $W_2^{2,1}(M_T)$ estimate, we obtain $W_p^{2,1}(M_T)$ a priori estimates for solutions of ($\ref{sys2}$) for all $ p>1$. For $a>0$ and smooth functions $f,g:M\times[0,\infty)\rightarrow \mathbb{R}$, Polden considered weighted inner products:
\[\langle f,g\rangle_{LL_a}=\int_0^{\infty} e^{-2at} \langle f(\cdot,t),g(\cdot,t)\rangle_{L^2(M)} dt\]
\[\langle f,g\rangle_{LW^1_a}=\int_0^{\infty} e^{-2at} \langle f(\cdot,t),g(\cdot,t)\rangle_{W_2^{1}(M)} dt\]
\[\langle f,g\rangle_{LW^2_a}=\int_0^{\infty} e^{-2at} \langle f(\cdot,t),g(\cdot,t)\rangle_{W_2^{2}(M)} dt\]
\[\langle f,g\rangle_{WW_a}=\langle f(\cdot,t),g(\cdot,t)\rangle_{LW^1_a}+\langle D_t f,D_tg\rangle_{LL_a}\]
Where $LL_a, LW_a$ and $WW_a$ are the Hilbert spaces formed by the completion of $C^{\infty}(M\times[0,\infty))$ in the corresponding norms, and $WW_a^0$ is the completion of subspace of  $C^{\infty}(M\times[0,\infty))$ with compact support in $WW_a$. See \cite{RefWorks:114} for the proof of the following result. 
\\
\begin{theorem}\label{1}
Suppose $\Psi_0$  lies in  $W_2^{1}(M)$ and $f\in LL_a(M\times[0,\infty))$. Then for sufficiently large a, the system $(\ref{sys2})$ has a  unique weak solution in  $WW_a^0$. 
\end{theorem} \\
\vspace{.2cm}\\
Furthermore using a priori estimates in \cite{RefWorks:114}, they showed that the solution belongs to $W^{2,1}_2(M\times[0,\infty))$.\\
\begin{theorem}\label{2}
Let $\Psi\in WW_a$ be the unique solution of $(\ref{sys2})$ with $\Psi_0\in W_2^{1}(M)$ and $f\in LL_a(M_T)$. Then $\Psi\in LW_a^2$, and there exists $C>0$ independent of $\Psi_0$ and $f$ such that 
\[\Vert \Psi\Vert_{LW_a^2}^2\leq C(\Vert \Psi_0\Vert_{W_2^1(M)}^2+\Vert f\Vert_{LL_a}^2)\]
\end{theorem}
\begin{proof} See Lemma 4.3 in \cite{RefWorks:114}.
\end{proof}\\

The result below is an immediate consequence.\\
\begin{corollary}\label{bd2}
Let $0<T<\infty$. Suppose $\Psi_0\in W_2^{1}(M)$ and $f\in L_2(M_T)$. Then there exists a unique weak solution to $(\ref{sys2})$ in $W^{2,1}_2(M_T)$, and there exists $C>0$ independent of $\Psi_0$ and $f$ such that  
\[\Vert \Psi\Vert_{W^{2,1}_2(M_T)}^2\leq C(\Vert \Psi_0\Vert_{W_2^1(M)}^2+\Vert f\Vert_{L_2(M_T)}^2)\]
\end{corollary}

We will use the $W_2^{2,1}(M_T)$ result to derive $W_p^{2,1}(M_T)$ a priori estimates for solutions to ($\ref{sys2}$) for all $ p>1$. To obtain these estimates, we transform the Cauchy problem defined locally on $M$ to a bounded domain on $\mathbb{R}^{n-1}$ and obtain the estimates  over this bounded domain. Then we pull the resulting estimates back to the manifold. Repeating this process over every neighborhood on the manifold, and using compactness of the manifold, we get estimates over the entire manifold. 

Let $\mathcal{F}$ be a subset of $\mathbb{R}_+$ with following property:\\
 $p>1$ belongs to $\mathcal{F}$ if and only if there exists $ C_{p,T}>0$ such that whenever $\Psi_0\in W^{2-\frac{2}{p}}_p(M)$ and $f\in L_p(M_T)$, then there exists a unique $\Psi\in W_p^{2,1}(M_T)$, such that $\Psi$ solves ($\ref{sys2}$) and 
\begin{eqnarray*}
\Vert \Psi\Vert_{p,M_T}^{(2)}\leq C_{p,T}(\Vert f\Vert_{p, M_T}+\Vert \Psi_0\Vert_{p,\p}^{(2-\frac{2}{p})})
\end{eqnarray*}
Note: From Corollary $\ref{bd2}$, $2\in\mathcal{F}$. Also note that we can prove Theorem $\ref{3}$ by showing $\mathcal{F}=(1,\infty)$.\\
\begin{lemma}\label{ap}
 $[2,\infty)\subset\mathcal{F}$.
\end{lemma}\\
\begin{proof}
We will show that  if $p\in \mathcal{F}$ then $[p, p+\frac{1}{n-1}]\subset\mathcal{F}$. To this end, let $p\in\mathcal{F}$ and $q\in [p, p+\frac{1}{n-1}]$ such that $\Psi_0\in W^{2-\frac{2}{q}}_{q}(M)$ and $ f\in L_{q}(M_T)$. Then $ f\in L_{p}(M_T)$ and $\Psi_0\in W^{2-\frac{2}{p}}_p(M)$. Since $p\in\mathcal{F}$, there exists $ C_{p,T}>0$ independent of $\Psi_0$ and $f$, and a unique $\Psi \in W^{2,1}_{p}(M_T)$ solving ($\ref{sys2}$) such that 
\begin{eqnarray}\label{p1}
\Vert \Psi\Vert_{p,M_T}^{(2)}&\leq& C_{p,T}(\Vert f\Vert_{p, M_T}+\Vert \Psi_0\Vert_{p,\p}^{(2-\frac{2}{p})})
\end{eqnarray}

Let $B(0,1)$ be the open ball in $\mathbb{R}^{n-1}$ of radius $1$ centered at the origin. Now, M is a $C^2$ manifold. Therefore, for each point $\xi\in\p$ there exists an open set $V_\xi$ of $\p$ containing $\xi$ and a $C^{2}$ diffeomorphism $\phi_\xi:B(0,1) \overset{\text{onto}}{\longrightarrow} V_\xi$.  Let ${\Phi}=\Psi \circ \phi_\xi$, $\tilde f=f \circ \phi_\xi$ and  ${\Phi_0}=\Psi_0 \circ \phi_\xi$. Using the Laplace Beltrami operator (defined in \cite{RefWorks:50}), $(\ref{sys2})$ takes the form\begin{align}\label{nbd}
 \Phi_t&={\frac{\tilde d}{\sqrt{det \ g}}}{\partial_{j}(g^{ij}\sqrt{det{\ g}}\ \partial_{i} \Phi)}+\tilde f(x,t) & x\in B(0,1),&\quad 0<t<T\nonumber\\
\Phi&={\Phi_0}&x\in B(0,1) ,&\quad t=0
\end{align}
where $g$ is the metric on $\p$ and $g^{ij}$ is the ${i,j}^{th}$ entry of the inverse of 
the matrix corresponding to metric $g$. That is, in the bounded region $B(0,1)\times (0,T)$, we have
\begin{align}\label{nodf}
\mathcal{ L}(\Phi) = \Phi_t - \displaystyle\sum\limits_{i,j=1}^{n-1} a_{ij}\Phi_{{x_{i}}{x_{j}}} + \displaystyle\sum\limits_{i=1}^{n-1} a_{i}\Phi_{x_{i}}&=\tilde f\\
{\Phi\big|}_{t=0}& = {\Phi_0}
\end{align}
where,\[a_{ij}=\tilde d \ g^{ij}\]
\[ a_{i}= {\frac{-\tilde d}{\sqrt{det\ g}}}{\partial_{j}(g^{ij}\sqrt{det{\ g}})}\]
Note $\Psi\in W^{2,1}_{p}(M_T)$ implies $\Phi\in W^{2,1}_{p}(B(0,1)\times(0,T)).$
Take $0<2r<1$ and define a cut off function $\psi\in C_{0}^{\infty}({\mathbb{R}}^{n-1},[0,1])$ such that,
\begin{align}\label{cut} \psi(x)=\begin{cases}
1 &   \forall  x\in B(0,r)\\
0 &   \forall x\in {\mathbb{R}}^{n-1} \backslash {B(0,2r)} \end{cases}\end{align} 
In $Q=B(0,2r)$, $Q_T=B(0,2r)\times(0,T)$ and $S_T=\partial B(0,r)\times(0,T)$, $w=\psi \Phi $ satisfies the equation
\begin{align*}
  \frac{\partial w }{\partial t} - \displaystyle\sum\limits_{i,j=1}^{n-1} a_{ij}\frac{\partial^{2}w}{\partial x_{i}\partial x_{j}}+\displaystyle\sum\limits_{i=1}^{n-1} a_{i}\frac{\partial w}{\partial x_{i}} &=\theta & (x,t)\in Q_T\\
 w&=0 &(x,t)\in S_T\\
{ w\big|}_{t=0} &=\psi {\Phi_0} & t=0, x \in Q
\end {align*}
where,\[ \theta=\tilde f\psi-2 \displaystyle\sum\limits_{i=1}^{n-1} a_{ij}\frac{\partial \Phi}{\partial x_{i}}\frac{\partial\psi }{\partial x_{j}}-\Phi\displaystyle\sum\limits_{i,j=1}^{n-1} a_{ij}\frac{\partial^{2}\psi }{\partial x_{i}\partial x_{j}}+\Phi\displaystyle\sum\limits_{i=1}^{n-1} a_{i}\frac{\partial\psi}{\partial x_{i}}\]
Since $\psi\in C_{0}^{\infty}({\mathbb{R}}^{n-1},[0,1])$ and $\Phi\in W^{2,1}_p(B(0,1)\times(0,T))$, therefore $\theta-\tilde f\psi\in W^{1,1}_{p}(Q_T)$.

Case 1.  Suppose $p<n$. From Lemma $\ref{i}$, $\theta-\tilde f\psi\in L_{\min\lbrace q,\ p+\frac{p^2}{n-p}\rbrace}(Q_T)$. In particular since $p+\frac{1}{n-1}<p+\frac{p^2}{n-p}$, and $\tilde f\psi\in L_{q}(Q_T)$, we have $\theta\in L_{q}(Q_T)$. As a result
\begin{eqnarray}
\nonumber \Vert\theta\Vert_{q, Q_T}&\leq& \Vert \tilde f\psi\Vert_{q, Q_T}+C_1 \Vert \Phi\Vert_{q, Q_T}+C_2 \Vert \Phi_x\Vert_{q, Q_T}\\
\nonumber &\leq& \Vert \tilde f\psi\Vert_{q, Q_T}+C_1 \Vert \Phi\Vert_{q, Q_T}+C_2 \Vert \Phi_{x}\Vert^{(1)}_{p, Q_T}
\end{eqnarray}
where $C_1, C_2>0$ are independent of $f$. Now in order to estimate $\Vert\Phi_x \Vert^{(1)}_{p,Q_T}$,  apply the change of variable \begin{eqnarray*}
\Vert \Phi_x\Vert_{p,Q_T}^{(1)}=\Vert \Psi_x |\det (({\phi_\xi^{-1}})^{'})| \Vert_{p,{(\phi_\xi(Q))}_T}^{(1)}\end{eqnarray*} and using $(\ref{p1})$,  we get
 \begin{eqnarray*}
\Vert\theta\Vert_{q, Q_T}\leq\Vert \tilde f\psi\Vert_{q, Q_T}+C_1 \Vert \Phi\Vert_{q, Q_T}+{C_2}_{p,T}(\Vert f\Vert_{p, M_T}+\Vert \Psi_0\Vert_{p,\p}^{(2-\frac{2}{p})})
\end{eqnarray*}
where ${C_2}_{p,T}>0$ is independent of $f$ and $\Psi_0$. At this point, we need an estimate on $\Vert \Phi\Vert_{q,Q_T}$. Again $\Vert \Phi\Vert_{q,Q_T}=\Vert \Psi |\det (({\phi_\xi}^{-1})^{'})|\Vert_{q,{(\phi_\xi(Q))}_T}$ and from  Lemma $\ref{i}$, \[\Vert \Psi |\det (({\phi_\xi}^{-1})^{'})|\Vert_{q,{(\phi_\xi(Q))}_T}\leq \tilde C\Vert \Psi |\det (({\phi_\xi}^{-1})^{'})|\Vert^{(1)}_{p,{(\phi_\xi(Q))}_T}\] Thus
 \begin{eqnarray}\label{hm}
\Vert\theta\Vert_{q, Q_T}\leq{K}_{p,T}(\Vert f\Vert_{p, M_T}+\Vert \Psi_0\Vert_{p,\p}^{(2-\frac{2}{p})})
\end{eqnarray}
where ${K}_{p,T}>0$ is independent of $f$ and $\Psi_0$.

Since $g_{i,j}$ are $C^1$ functions on the compact manifold $M$,  $a_{i,j }$ and $a_{i}$ satisfy the hypothesis (bounded continuous function in $\overline{Q_T})$ of Lemma $\ref{L1}$. Therefore using Lemma $\ref{L1}$, 
 \begin{eqnarray}\label{eq88}
 \Vert w\Vert_{q,Q_T}^{(2)}&\leq&  C_{q,T}(\Vert \theta\Vert_{q,Q_T}+\Vert \psi {\Phi_0}\Vert_{q,Q}^{(2-\frac{2}{q})})
 \end{eqnarray}
where $C_{q,T}>0$ is independent of $\theta$ and $\psi\Phi_0$.
Combining $(\ref{hm})$ and $(\ref{eq88})$ we get,\begin{eqnarray*}
 \Vert w\Vert_{q,Q_T}^{(2)}\nonumber&\leq&  C_{q,T}(\Vert \theta\Vert_{q,Q_T}+\Vert \psi \Phi_0\Vert_{q,Q}^{(2-\frac{2}{q})})\\\nonumber
 &\leq&  {\tilde K}_{p,T}(\Vert f\Vert_{p, M_T}+\Vert \Psi_0\Vert_{p,\p}^{(2-\frac{2}{p})}+\Vert \psi \Phi_0\Vert_{q,Q}^{(2-\frac{2}{q})})
 \end{eqnarray*}
where $\tilde K_{p,T}>0$ is independent of $f$, $\theta$ and $\psi\Phi_0$. Note that $w=\Phi$ on $W_T=B(0,r)\times(0,T)$. Thus
 \begin{align}\label{eq33}
 \Vert \Phi\Vert_{q,W_T}^{(2)}\leq   {\tilde K}_{p,T}(\Vert f\Vert_{p, M_T}+\Vert \Psi_0\Vert_{p,\p}^{(2-\frac{2}{p})}+\Vert \psi \Phi_0\Vert_{q,Q}^{(2-\frac{2}{q})})
 \end{align}
Observe $(\ref{eq33})$ is over $B(0,r)\times(0,T)\subset\mathbb{R}^{n-1}\times \mathbb{R}_+$. To get the estimate back on the manifold, apply the change of variable, $\Vert \Phi\Vert_{q,W_T}^{(2)}=\Vert \Psi |\det (({\phi^{-1}})^{'})| \Vert_{q,\phi ( W_T)}^{(2)}$ and using  first mean value theorem of integration there exist $\hat\xi\in \phi(W_T)$, and $\tilde K_{p,T,\hat\xi}$ such that \begin{eqnarray}\label{eqr}
  \Vert \Psi\Vert_{q,\phi(W_T)}^{(2)}&\leq&  {\tilde K}_{p,T,\hat\xi}(\Vert f\Vert_{p, M_T}+\Vert \Psi_0\Vert_{p,\p}^{(2-\frac{2}{p})}+\Vert  \Psi_0\Vert_{q,\phi(Q)}^{(2-\frac{2}{q})})
 \end{eqnarray}

 So far, an estimate in one open neighborhood of some point $\xi \in\p$ is obtained.
As one varies the point $\xi$  on $\p$, there exist corresponding open neighborhoods $V_\xi$ and a  smooth diffemorphisms $\phi_{\xi}:B(0,r){\longrightarrow} V_\xi$, which results in different $\tilde K_{p,T,\hat\xi}$ for every $V_\xi$. Consider an open cover of $\p$ such that $\p=\bigcup_{\substack{\xi\in\p}}V_{\xi}$. 
Since $\p$ is compact, there exists $\lbrace \xi_1,\xi_2,...,\xi_N \rbrace$ such that  $M\subset\bigcup_{\substack{\xi_j\in M\\ 1\leq j\leq N}}V_{\xi_j}$ and $\tilde K_{p,T,\hat\xi_j}$ corrresponding to each $V_{\xi_j}$. 
Let, $ C_{p,M,T}= \sum_{\substack{1\leq j\leq N}} \tilde K_{p,T,\hat\xi_j}$. Inequality $(\ref{eqr})$ implies

\begin{eqnarray*}
 \Vert \Psi\Vert_{q,M_T}^{(2)}&\leq& C_{p,M,T}(\Vert f\Vert_{q, M_T}+\Vert \Psi_0\Vert_{q,M}^{(2-\frac{2}{q})})
 \end{eqnarray*}
Thus $[p,p+\frac{1}{n-1}]\subset\mathcal{F}$.

Case 2. Suppose $p\geq n$.

 By Lemma $\ref{i}$ and Theorem 4.12 in \cite{RefWorks:76}, if $q\in[p,\infty)$, $\Psi_0\in W^{2-\frac{2}{q}}_q(M)$, and $f\in L_q(M_T)$ then $\theta\in L_{ q}(Q_T)$, and proceeding similarly to Case 1, we get
\begin{eqnarray*}
 \Vert \Psi\Vert_{q,M_T}^{(2)}&\leq&  C_{q,M,T}(\Vert f\Vert_{q, M_T}+\Vert \Psi_0\Vert_{q,M}^{(2-\frac{2}{q})})
 \end{eqnarray*}
where $C_{q,M,T}>0$ is independent of $f$, $\theta$ and $\psi\Phi_0$. Hence $[2,  \infty)\subset\mathcal{F}$.\end{proof}\\
\\
{\bf Proof of Theorem $\ref{3}$:} From Lemma $\ref{ap}$, we have $[2,  \infty)\subset\mathcal{F}$. It remains to show that $(1,2)\subset\mathcal{F}$. Let $1<p<2$ , $f\in L_p(M_T)$ and $\Psi_0\in W^{2-\frac{2}{p}}(M)$. Since $C^\infty(\overline {M_T})$ is dense in $L_p(M_T)$ and $C^{\infty}(\overline M)$ is dense in $W_p^{2-\frac{2}{p}}(M)$, there exist a sequences of functions $\lbrace{f_k\rbrace}\subseteq C^{\infty}(\overline {M_T})$ and $\lbrace{{\Psi_0}_k\rbrace}\subseteq C^{\infty}(\overline M)$ such that $f_k$ converges to $ f$ in $L_p(M_T)$ and ${\Psi_0}_k$ converges to $\Psi_0$ in $W_p^{2-\frac{2}{p}}(M)$. Define a sequence $\lbrace\Psi_k\rbrace$ such that, 
\begin{align}\label{k}
 \Psi_{k_t }\nonumber&= \tilde d \Delta_{\p} \Psi_k+f_k &\xi\in \p,\quad & 0<t<T\\ 
 {\Psi_k} &= {\Psi_0}_k\hspace{1.2in}  &\xi\in\p ,\quad & t=0
\end{align}
Now, transform system $(\ref{k})$ over a bounded region in $\mathbb{R}^{n-1}$. Similar to the proof of Lemma $\ref{ap}$, for each point $\xi\in M$ there exists an open set $V_{\xi}$ of $M$ containing $\xi$ and a $C^2$ diffeomorphism $\phi_{\xi}:B(0,1)\overset{\text{onto}}{\longrightarrow} V_{\xi}$. Corresponding to each $k$, let $\tilde f_k=f_k\circ \phi_\xi$, ${\Phi_0}_k={\Psi_0}_k \circ \phi_\xi$ and using the Laplace Beltrami operator, $(\ref{k})$ on $B(0,1)\subset U$ takes the form
\begin{align}\label{rk}
 {\Phi_k}_t&={\frac{\tilde d}{\sqrt{det \ g}}}{\partial_{j}(g^{ij}\sqrt{det\ {g}}\ \partial_{i} \Phi_k)}+\tilde f_k & x\in B(0,1),&\quad 0<t<T\\
\Phi_k&={\Phi_0}_k\hspace{2.4in}  &x\in B(0,1) ,&\quad t=0\nonumber
\end{align}

Consequently, in a bounded region $B(0,1)\times (0,T)$ of the Euclidean space, we consider ($\ref{rk}$) in the nondivergence form defined in ($\ref{nodf}$) for each $\Phi_k$, with $\tilde f$ replaced by $\tilde f_k$ and $\Phi_0$ by ${\Phi_0}_k$. 

Taking $0<2r<1$, using a cut off function $\psi\in C_{0}^{\infty}({\mathbb{R}}^{n-1},[0,1])$ defined in $(\ref{cut})$, and defining $Q=B(0,2r)$, $Q_T=B(0,2r)\times(0,T)$, $S_T=\partial B(0,r)\times(0,T)$, and $w_k=\psi \Phi_k $, we see that
\begin{align*}
  \frac{\partial w_k }{\partial t} - \displaystyle\sum\limits_{i,j=1}^{n-1} a_{ij}\frac{\partial^{2}w_k}{\partial x_{i}\partial x_{j}}+\displaystyle\sum\limits_{i=1}^{n-1} a_{i}\frac{\partial w_k}{\partial x_{i}} &=\theta_k & (x,t)\in Q_T\\
 w_k&=0 &(x,t)\in S_T\\
{ w_k\big|}_{t=0} &=\psi {\Phi_0}_k & t=0,  x \in Q
\end {align*}where,
\[\theta_k=\tilde f_k\psi -2 \displaystyle\sum\limits_{i=1}^{n-1} a_{ij}\frac{\partial \Phi_k}{\partial x_{i}}\frac{\partial\psi }{\partial x_{j}}-\Phi_k\displaystyle\sum\limits_{i,j=1}^{n-1} a_{ij}\frac{\partial^{2}\psi }{\partial x_{i}\partial x_{j}}+\Phi_k\displaystyle\sum\limits_{i=1}^{n-1} a_{i}\frac{\partial\psi}{\partial x_{i}}\]
Note that $f_k$ and ${\Psi_0}_k $ are smooth functions. Therefore Lemma $\ref{ap}$ guarantees $\Phi_k\in W^{2,1}_q(Q_T)$ for all $q\geq2$. Thus $\theta_k\in L_q(Q_T)$ for all $q\geq 2$. 
Recall $\psi\in C_{0}^{\infty}({\mathbb{R}}^{n-1},[0,1])$. Using Lemma $\ref{L4}$ for $\epsilon>0$, there exists $c_{\epsilon}>0$ such that
\begin{align}\label{eq5}
\nonumber \Vert\theta_k\Vert_{p, Q_T}&\leq\Vert \tilde f_k\psi\Vert_{p, Q_T}+M_1 \Vert \Phi_k\Vert_{p, Q_T}+ M_2 \Vert {\Phi_k}_x\Vert_{p, Q_T}\\
\nonumber &\leq  \Vert \tilde f_k\Vert_{p, Q_T}+M_1 \Vert \Phi_k\Vert_{p, Q_T}\\
 &\quad+M_2 (\epsilon\Vert {\Phi_k}_{xx}\Vert_{p, Q_T}+c_{\epsilon} \Vert \Phi_k\Vert_{p, Q_T})\quad\quad
\end{align}
Here $M_1, M_2>0$ are independent of $f$ and $\Psi_0$. At this point we need an estimate for $\Vert \Phi_k\Vert_{p,Q_T}$. From Lemma $\ref{i}$ for $1<p \leq n<q$ there exists $C_{\epsilon}>0$ such that 
\begin{align*}
{\Vert \Phi_k\Vert}^{p}_{L_{\frac{pq}{q-p}}(Q_T)}\nonumber &\leq \epsilon ({\Vert {\Phi_k}_{x}\Vert}^{p}_{p, Q_T}+ {\Vert {\Phi_k}_{t}\Vert}^{p}_{p, Q_T} ) + C_{\epsilon}{\Vert \Phi_k\Vert}^{p}_{{1, Q_T}}
\end{align*}
Since $p<\frac{pq}{q-p}$,  from H\"{o}lder's inequality, $\epsilon$ and $C_\epsilon$ get scaled to $\tilde \epsilon>0$ and $C_{\tilde \epsilon}>0$ (with $\tilde \epsilon\rightarrow 0^+$ as $\epsilon\rightarrow 0^+$), and
\begin{align}\label{eq6}
 \Vert \Phi_k\Vert_{p,Q_T}&\leq \tilde\epsilon( \Vert {\Phi_k}_t\Vert_{p,Q_T}+ \Vert {\Phi_k}_x\Vert_{p,Q_T})\\ \nonumber
&\quad+C_{\tilde\epsilon}{\Vert \Phi_k\Vert}_{{1,Q_T} }\hspace{.8in}
\end{align}
From $(\ref{eq5})$ and $(\ref{eq6})$,
\begin{align}
  \Vert\theta_k\Vert_{p, Q_T} \nonumber&\leq (M_1+M_2 c_{\epsilon}) ( \tilde\epsilon( \Vert {\Phi_k}_t\Vert_{p,Q_T}+ \Vert {\Phi_k}_{x}\Vert_{p,Q_T})+C_{\tilde\epsilon}{\Vert \Phi_k\Vert}_{{1, Q_T }})\\ \nonumber
&\quad+ {\Vert \tilde f_k\Vert_{p, Q_T}}+M_2 \epsilon{\Vert {\Phi_k}_{xx}\Vert_{p, Q_T}}
\end{align}

Recall $g_{i,j}$ are $C^1$ functions on the compact manifold $M$. Therefore $a_{i,j }$ and $a_{i}$ satisfy the hypothesis (bounded continuous function in $\overline{Q_T})$ of Lemma $\ref{L1}$. Using Lemma $\ref{L1}$ for $ p\ne\frac{3}{2}$, 
 \begin{align}\label{eq8}
 \Vert w_k\Vert_{p,Q_T}^{(2)}&\leq C_{p,T}(\Vert \theta_k\Vert_{p,Q_T}+\Vert \psi {\Phi_0}_k\Vert_{p,Q}^{(2-\frac{2}{p})})
 \end{align}
where $C_{p,T}$ is independent of $\theta$ and $\psi\Phi_0$.
Combining $(\ref{eq5})$ and $(\ref{eq8})$, we get \begin{align*}
 \Vert w_k\Vert_{p,Q_T}^{(2)}\nonumber&\leq  C_{p,T}(\Vert \theta_k\Vert_{p,Q_T}+\Vert \psi {\Phi_0}_k\Vert_{p,Q}^{(2-\frac{2}{p})})\\\nonumber
 &\leq  C_{p,T}\lbrace\Vert \tilde f_k\Vert_{p, Q_T}+M_2 \epsilon\Vert {\Phi_k}_{xx}\Vert_{p, Q_T}\\ \nonumber&\quad +(M_1+M_2 c_{\epsilon}) ( \tilde\epsilon( \Vert {\Phi_k}_t\Vert_{p,Q_T}+ \Vert {\Phi_k}_x\Vert_{p,Q_T})+C_{\tilde\epsilon}{\Vert \Phi_k\Vert}_{{1, Q_T }})\\
&\quad+\Vert \psi {\Phi_0}_k\Vert_{p,Q}^{(2-\frac{2}{p})}\rbrace
\end{align*}
Note that $w_k=\Phi_k$ on $W_T=B(0,r)\times(0,T)$. Thus
\begin{align}\label{eq3}
 \Vert \Phi_k\Vert_{p,W_T}^{(2)} \nonumber\leq  C_{p,T}\lbrace &\Vert \tilde f_k\Vert_{p, Q_T}+M_2 \epsilon\Vert {\Phi_k}_{xx}\Vert_{p, Q_T}\\ \nonumber& +(M_1+M_2 c_{\epsilon}) ( \tilde\epsilon( \Vert {\Phi_k}_t\Vert_{p,Q_T}+ \Vert {\Phi_k}_x\Vert_{p,Q_T})+C_{\tilde\epsilon}{\Vert \Phi_k\Vert}_{{1, Q_T }})\\
&+\Vert \psi {\Phi_0}_k\Vert_{p,Q}^{(2-\frac{2}{p})}\rbrace
 \end{align}
Observe $(\ref{eq3})$ is over $B(0,r)\times(0,T)\subset\mathbb{R}^{n-1}\times \mathbb{R}_+$. To get an estimate on the manifold, apply the change of variable, $\Vert \Phi_k\Vert_{p,W_T}^{(2)}=\Vert \Psi_k |\det (({\phi^{-1}})^{'})| \Vert_{p,\phi ( W_T)}^{(2)}$  and using  first mean value theorem of integration there exist $\hat\xi\in \phi(W_T)$, and $\tilde C_{p,T,\hat\xi}$ such that
 \begin{align}\label{eq9}
  \Vert {\Psi_k}\Vert_{p,\phi(W_T)}^{(2)} \nonumber&\leq \tilde C_{p,\xi,T} \lbrace\Vert  f_k\Vert_{p, \phi( Q_T)}+M_2 \epsilon\Vert {\Psi_k}_{xx}\Vert_{p, \phi( Q_T)}\\ \nonumber&\quad+(M_1+M_2 c_{\epsilon}) ( \tilde\epsilon( \Vert {\Psi_k}_t\Vert_{p,\phi( Q_T)}+ \Vert {\Psi_k}_x\Vert_{p,\phi( Q_T)})+C_{\tilde\epsilon}{\Vert {\Psi_k}\Vert}_{{1,(\phi( Q_T))} })\\
&\quad+\Vert  {\Psi_0}_k\Vert_{p,\phi(Q)}^{(2-\frac{2}{p})}\rbrace
 \end{align}
 So, an estimate in an open neighborhood of a point $\xi \in\p$ can be obtained.
As one varies the point $\xi$  on $\p$, there exist corresponding open neighborhoods $V_\xi$ and a  smooth diffemorphisms $\phi_{\xi}:B(0,r)\longrightarrow V_\xi$, which result in different $\tilde C_{p,\hat\xi,T}$ for every $V_\xi$. Consider an open cover of $\p$ such that $\p=\bigcup_{\xi\in\p}V_{\xi}$.
Since $\p$ is compact, there exists $\lbrace \xi_1,\xi_2,...,\xi_N \rbrace$ such that  $M\subset\bigcup_{\substack{\xi_j\in M\\ 1\leq j\leq N}}V_{\xi_j}$ and $\tilde C_{p,\hat\xi_j,T}$ corrresponding to each $V_{\xi_j}$. Let $\hat C_{p,T}= \sum_{\substack{1\leq j\leq N}} \tilde C_{p,\hat\xi_j,T}$. Inequality $(\ref{eq9})$ implies
 \begin{align}\label{eq2}
  \Vert {\Psi_k}\Vert_{p,M_T}^{(2)} \nonumber&\leq \hat C_{p,T} \lbrace\Vert  f_k\Vert_{p, M_T}+M_2 \epsilon\Vert {\Psi_k}_{xx}\Vert_{p, M_T}\\ \nonumber&\quad+(M_1+M_2 c_{\epsilon}) ( \tilde\epsilon( \Vert {\Psi_k}_t\Vert_{p,M_T}+ \Vert {\Psi_k}_x\Vert_{p,M_T})+C_{\tilde\epsilon}{\Vert {\Psi_k}\Vert}_{{1,M_T} })\\
&\quad+\Vert  {\Psi_0}_k\Vert_{p,M}^{(2-\frac{2}{p})}\rbrace
 \end{align}
Also, a simple calculation gives
\begin{align*}
\Vert {\Psi_k}\Vert_{1,M_T}&\leq\Vert  f_k\Vert_{1,M_T}+ \Vert {\Psi_0}_k\Vert_{1,M}
\end{align*}
 Now, choose $\epsilon>0$ such that,

\[\max\lbrace\hat C_{p,T}M_2\epsilon ,\quad  \hat C_{p,T}\tilde\epsilon(M_1+M_2c_{\epsilon})\rbrace<\frac{1}{2}\]
For this choice of $\epsilon$, $(\ref{eq2})$  gives the $W^{2,1}_p$ estimates
\begin{align*}
 \Vert {\Psi_k}\Vert_{p,M_T}^{(2)}&\leq\hat C_{p,T}(\Vert f_k\Vert_{p, M_T}+C_\epsilon (\Vert  f_k\Vert_{1,M_T}+ \Vert {\Psi_0}_k\Vert_{1,M})+\Vert  {\Psi_0}_k\Vert_{p,M}^{(2-\frac{2}{p})})
 \end{align*}
\begin{align}\label{seqk}
 \Vert {\Psi_k}\Vert_{p,M_T}^{(2)}&\leq\hat  K_{p,T}(\Vert f_k\Vert_{p, M_T}+\Vert  {\Psi_0}_k\Vert_{p,M}^{(2-\frac{2}{p})})\end{align}
where $\hat K_{p,T}>0$ is independent of $f_k$ and ${\Psi_0}_k$. It remains to show that the sequence $\lbrace{\Psi_k\rbrace}$ converges to a function $\Psi$ in $W^{2,1}_p(M_T)$, and $\Psi$ solves $(\ref{sys2})$. From linearity and $(\ref{seqk})$,  if $m,l\in \mathbb{N}$ then  $\Psi_m-\Psi_l$ satisfies
\begin{align*}
 ({\Psi_m} -{\Psi_l})_t &= \tilde d \Delta_{\p} (\Psi_m-\Psi_l)+f_m-f_l & \xi\in \p,\quad &0<t<T\\ \nonumber
 \Psi_m-\Psi_l& = {\Psi_0}_m-{\Psi_0}_l&  \xi\in\p ,\quad& t=0
\end{align*}
and
\begin{align*}
 \Vert \Psi_m-\Psi_l\Vert_{p, M_T}^{(2)}&\leq\hat K_{p,T}(\Vert f_m-f_l\Vert_{p, M_T}+\Vert{\Psi_0}_m-{\Psi_0}_l\Vert_{p,M}^{(2-\frac{2}{q})})
 \end{align*}
This implies $\lbrace{ \Psi_k\rbrace}$ is a Cauchy sequence in $W^{2,1}_p(M_T)$, so there is a function $\psi\in W^{2,1}_p(M_T)$ such that $\Psi_k\rightarrow\Psi$. Then $f_k$ converges to $f$ in $ L_p(M_T)$, ${\Psi_0}_k$ converges to $\Psi_0$ in $W_p^{2-\frac{2}{p}}(M)$, and $\Psi_k$ converges to $\Psi\in W^{2,1}_p(M_T)$. Therefore $\Psi$ solves $(\ref{sys2})$, and $(\ref{seqk})$ implies \begin{align*}
 \Vert \Psi\Vert_{p,M_T}^{(2)}&\leq\hat  K_{p,T}(\Vert f\Vert_{p, M_T}+\Vert \Psi_0\Vert_{p,M}^{(2-\frac{2}{p})})
 \end{align*} Hence
 $\mathcal{F}=(1,\infty)$, and the proof of Theorem $\ref{3}$ is complete.

\section{H\"{o}lder Estimates for the Neumann problem}
\setcounter{equation}{0}

The following result is a version of Theorem 9.1 with Neumann boundary conditions, referred to in chapter 4 of \cite{RefWorks:65} on page 351. 
\begin{lemma}\label{L1.5}
Let p $ > 1$. Suppose $ \theta\in L_{p}{(\Omega\times(0, T))}$, $\varphi_0\in W_p^{(2-\frac{2}{p})}(\Omega)$ and $\gamma\in W_{p}^{1-\frac{1}{p},\frac{1}{2}-\frac{1}{2p}}(M\times(0, T))$ with $p\neq 3$ . In addition, when $p>3$ assume \[  d\frac{\partial {\varphi_0}}{\partial \eta} =\gamma\quad{\text {on $M\times\lbrace 0\rbrace$}}\] Then $(\ref{m2})$ has a unique solution $ \varphi\in W_{p}^{2,1}{(\Omega\times(0, T))}$ and there exists $C$ dependent upon $\Omega, p,  T$, and independent of $\theta, \varphi_0$ and $\gamma$ such that
\[{\Vert \varphi\Vert}_{p,(\Omega\times(0, T))}^{(2)}\leq C( {\Vert \theta\Vert}_{p,(\Omega\times(0, T))}+{\Vert \varphi_0\Vert}_{p,\Omega}^{(2-\frac{2}{p})}+{\Vert \gamma\Vert}_{p,(\partial\Omega\times(0, T))}^{(1-\frac{1}{p},\frac{1}{2}-\frac{1}{2p})})\]\end{lemma}

\begin{definition}
$\varphi$ is said to be a weak solution of system $(\ref{m2})$ from $V_2^{1,\frac{1}{2}}(\Omega_{ T})$ if and only if \begin{align*}
-\int_0^{T} \int_\Omega \varphi \nu_t &- \int_0^{ T} \int_{\partial\Omega} d\ \nu\frac{\partial \varphi}{\partial\eta}+\int_0^{T} \int_\Omega d\ \nabla\nu.\nabla\varphi-\int_0^{ T} \int_\Omega\theta\nu\ \\
&= \int_{\Omega}\nu(x,0)\varphi(x,0)\end{align*}for any $\nu \in W_2^{1,1}(\Omega_{ T})$ that is equal to zero for $t= T$. 
\end{definition}\\

We also need a notion of solution of ($\ref{sy3}$) which was first introduced in the study of Dirichlet and Neumann problems for the Laplace operator in a bounded $C^1$ domain by Fabes, Jodeit and Rivier \cite{RefWorks:106}. They used Calderon's result in \cite{RefWorks:107} on $L^p$ continuity of Cauchy integral operators for $C^1$ curves. Further in \cite{RefWorks:105},  Fabes and Riviere constructed solutions to the initial Neumann problem for the heat equation satisfying the zero initial condition in the form of a single layer heat potential, when densities belong to $L_p(M\times (0,T))$, $1<p<\infty$.  We will consider the solution to $(\ref{sy3})$ in the sense of one which is constructed in \cite{RefWorks:105}. 

The following result plays a crucial role for that construction of solution to make sense, and is proved in \cite{RefWorks:105}.\\

\begin{proposition} Assume $\Omega$ is a $C^1$ domain and for $ Q\in{M}$, $\eta_Q$ being the unit outward normal to $M$ at $Q.$ For $0<\epsilon<t$ set \[J_{\epsilon}(f)(Q,t)=  \int_0^{t-\epsilon}\int_{M}\frac{\langle y-Q,\eta_Q\rangle}{(t-s)^{\frac{n}{2}+1}}\exp \left(-\frac{\vert Q-y\vert^2}{4(t-s)}\right)f(s,y)\ d\sigma \ ds\] Then
\begin{enumerate}
\item For every $1<p<\infty$ there exists $C_p>0$ such that $\sup_{0<\epsilon<t}\vert J_{\epsilon}(f)(Q,t)\vert= J(f)(Q,t)$ satisfies \[\Vert  J (f)\Vert_{L_{p}(M\times(0, T))}\leq C_p\Vert f\Vert_{L_{p}(M\times(0, T))} \text{ for all } f\in L_p(M\times (0,  T))\]
\item $ \lim_{\epsilon\rightarrow 0^{+}}J_{\epsilon}(f)=J(f)$ exists in $L_p(M\times(0, T))$ and pointwise for almost every $(Q,t)\in (M\times(0, T))$ provided $f\in L_p(M\times(0, T)), 1<p<\infty$.
\item $ c_n I+J$ is invertible on  $L_p(M\times(0, T))$ for each $1<p<\infty$ and $c_n\neq 0$.
\end{enumerate}
\end{proposition}

We consider the case $d=1$ below. The extension to arbitrary $d>0$ is straightforward. For $Q\in M$, $(x,t)\in \Omega_T$ and $t>s$, consider\[W(t-s,x,Q)=\frac{\exp\left(\frac{-|x-Q|^2}{4(t-s)}\right)}{(t-s)^\frac{n}{2}}\  \text {and } g(Q,t)=-2 [-c_n I +J]^{-1}\gamma(Q,t)\]
where $c_n$ is given in $\cite{RefWorks:105}$.
\begin{definition}
$\varphi$ is said to be a classical solution of system $(\ref{sy3})$ with $ d=1$ and, $\gamma\in L_p(M\times(0,T))$ for $p>1$ if and only if
 \[\varphi(x,t)= \int_0^t\int_{M} W(t-s,x,Q)g(Q,s)\ d\sigma\ ds \text { for all }\ (x,t)\in\Omega_T\]
\end{definition}

\begin{rems}
When $\theta=0$ and $\varphi (x,0)=0$, the weak solution of $(\ref{m2})$  is the same as the classical solution of $(\ref{sy3})$. 
\end{rems}\\
\\
In order to prove the classical solution $\varphi$ to $(\ref{sy3})$  is H\"{o}lder continuous, let $(x,T)$, $(y,\tau)\in \Omega_T$ such that 
\[\varphi(x,T)= \int_0^T\int_{M} W(T-s,x,Q)g(Q,s)\ d\sigma\ ds\] and \[\varphi(y,\tau)= \int_0^ \tau\int_{M} W(\tau-s,y,Q)g(Q,s)\ d\sigma\ ds\] Without loss of generality we assume $0<\tau<T$. Consider the difference \begin{align*}
\varphi(x,T)-\varphi(y,\tau)&=\int_0^\tau\int_{M} (W(T-s,x,Q)-W(\tau-s,y,Q))g(Q,s)\ d\sigma \ ds\\
&\quad\quad+\int_\tau^T\int_{M} W(T-s,x,Q)g(Q,s)\ d\sigma \ ds
\end{align*}
 Lemmas $\ref{a}$, $\ref{e}$ and $\ref{c}$ provide estimates needed to prove $\varphi$ is H\"{o}lder continuous. Throughout the proofs we assume $p'=\frac{p}{p-1}$.\\

\begin{lemma}\label{a}
Let $p>n+1$. Suppose $(x,T)$, $(y,\tau)\in\Omega_{ T}$  with $0<\tau<T$ and $ \R^c= \lbrace( Q,s)\in M\times(0,\tau):|x-Q|+|T-s|^{\frac{1}{2}}<2(|x-y|+|T-\tau|^\frac{1}{2}) \rbrace$. Then for $ 0<a<1-\frac{n+1}{p}$ there exists $K_1>0$ depending on $p, n, \overline\Omega$, $T$ and independent of $g\in L_p(M\times(0,T))$  such that
\begin{align*}
\int_{\R^c}|(W(T-s,x,Q)&-W(\tau-s,y,Q))g(Q,s)|\ d\sigma\ ds\\ &\leq  K_1\left( |x-y|+|T-\tau|^\frac{1}{2}\right)^a \parallel g\parallel_{p,M\times[0,\tau]}
\end{align*}
\end{lemma}
\begin{proof}
\begin{align*}
\int_{\R^c}|(W(T-s,x,Q)&-W(\tau-s,y,Q))g(Q,s)|\ d\sigma\ ds\\
&=\int_{\R^c}\left| \frac{\exp \left(\frac{-|x-Q|^2}{4(T-s)}\right)}{(T-s)^{\frac{n}{2}}}-\frac{\exp \left(\frac{-|y-Q|^2}{4(\tau-s)}\right)}{(\tau-s)^{\frac{n}{2}}}\right||g(Q,s)|\ d\sigma\ ds\\
&\leq \left[\left(\int_{\R^c}\left( \frac{\exp \left(\frac{-|x-Q|^2}{4(T-s)}\right)}{(T-s)^{\frac{n}{2}}}\right)^{p'}\right)^{\frac{1}{p'}}+\left(\int_{\R^c}\left( \frac{\exp \left(\frac{-|y-Q|^2}{4(\tau-s)}\right)}{(\tau-s)^{\frac{n}{2}}}\right)^{p'}\right)^{\frac{1}{p'}}\right]\Vert  g\Vert_{p,\R^c}
\end{align*}
By hypothesis $p>n+1$. Pick $0<\epsilon<\frac{p-(n+1)}{p-1}$, set $N=\frac{n-1-\epsilon}{2}$. Then there exists $c>0$ such that $w^N\cdot\exp(-w)\leq c\cdot N$ for all $w\geq 0$. Consequently,
\begin{align*}
& \left[\left(\int_{\R^c}\frac{\exp \left(\frac{-p'|x-Q|^2}{4(T-s)}\right)}{(T-s)^{\frac{np'}{2}}}\right)^{\frac{1}{p'}}+\left(\int_{\R^c}\frac{\exp \left(\frac{-p'|y-Q|^2}{4(\tau-s)}\right)}{(\tau-s)^{\frac{np'}{2}}}\right)^{\frac{1}{p'}}\right] \Vert  g\Vert_{p,\R^c}\\
&\leq \left[\left(\int_{\R^c} \frac{c\cdot N}{(T-s)^{\frac{np'}{2}}\left(\frac{p'|x-Q|^2}{4(T-s)}\right)^N}\right)^{\frac{1}{p'}}+\left(\int_{\R^c} \frac{c\cdot N}{(\tau-s)^{\frac{np'}{2}}\left(\frac{p'|y-Q|^2}{4(\tau-s)}\right)^N}\right)^{\frac{1}{p'}}\right] \Vert g\Vert_{p,\R^c}\\
&\leq \left[C_1 \left(\int_0^\tau{(T-s)^{\frac{n-1-\epsilon-np'}{2}}}ds\int_A\frac{1}{|x-Q|^{n-1-\epsilon}}\ d\sigma\right)^\frac{1}{p'} \right. \\
&\left. \quad + C_2 \left(\int_0^\tau{(\tau-s)^{\frac{n-1-\epsilon -np'}{2}}}ds\int_A\frac{1}{|y-Q|^{n-1-\epsilon}}\ d\sigma\right)^\frac{1}{p'}\right] \parallel g\parallel_{p,\R^c}
\end{align*}
where $A= \lbrace Q\in M:|x-Q|<2|x-y|+|T-\tau|^{\frac{1}{2}} \rbrace$. Since $|T-\tau|<|T-s|, \R^c \subset  A\times(0,\tau)$. Let $\rho_y=|y-Q|$, $\rho_x=|x-Q|$. Notice that in $A,$ $0<\rho_x<2|x-y|+|T-\tau|^\frac{1}{2}$ and $0<\rho_y<|x-y|+\rho_x <3|x-y|+|T-\tau|^\frac{1}{2}$. Therefore,
\begin{align*}
& \left[ C_1 \left(\int_0^\tau{(\tau-s)^{\frac{n-1-\epsilon -np'}{2}}}ds\int_A\frac{1}{|y-Q|^{n-1-\epsilon}}d\sigma\right)^\frac{1}{p'}\right. \\
&\left. +C_2 \left(\int_0^\tau{(T-s)^{\frac{n-1-\epsilon-np'}{2}}}ds\int_A\frac{1}{|x-Q|^{n-1-\epsilon}}d\sigma\right)^\frac{1}{p'}\right]\parallel g\parallel_{p,\R^c}\\
&\leq \left[\tilde C_1 \left(\int_0^\tau{(\tau-s)^{\frac{n-1-\epsilon -np'}{2}}}ds\int_0^{3|x-y|+|T-\tau|^\frac{1}{2}}r^{\epsilon-1}dr\right)^\frac{1}{p'}\right.\\
&\left. +\tilde C_2 \left(\int_0^\tau{(T-s)^{\frac{n-1-\epsilon-np'}{2}}}ds\int_0^{2|x-y|+|T-\tau|^\frac{1}{2}}r^{\epsilon-1}dr\right)^\frac{1}{p'}\right]\Vert g\Vert_{p,\R^c}
\end{align*}
\begin{align*}
&\leq \left[\frac{\tilde C_1}{\epsilon^\frac{1}{p'}} {(\tau)^{\frac{n+1-\epsilon -np'}{2p'}}}\left(3|x-y|+|T-\tau|^\frac{1}{2}\right)^\frac{\epsilon}{p'}\right. \\
&\left.  +\frac{\tilde C_2}{\epsilon^\frac{1}{p'}} \left({T^{\frac{n+1-\epsilon -np'}{2}}-(T-\tau)^{\frac{n+1-\epsilon -np'}{2}}}\right)^\frac{1}{p'}\left(2|x-y|+|T-\tau|^\frac{1}{2}\right)^\frac{\epsilon}{p'}\right]\Vert g\Vert_{p,\R^c}
\end{align*}
By hypothesis, $p'<\frac{n+1-\epsilon}{n}$. Therefore, there exists $K_1>0$ depends on $p, n$ and $T$ such that
\begin{align*}
\int_{\R^c}|(W(T-s,x,Q)&-W(\tau-s,y,Q))g(Q,s)|\ d\sigma\ ds\\ &\leq  K_1\left( |x-y|+|T-\tau|^\frac{1}{2}\right)^\frac{\epsilon(p-1)}{p}\parallel g\parallel_{p,M\times[0,\tau].}
\end{align*}  The result follows since $0<\epsilon<\frac{ap}{p-1}$ is arbitrary.\end{proof}\\
\\
The proof of the following Lemma makes use of Brown's corollary to Theorem 3.1 in \cite{RefWorks:81}. This also provides a proof for the remark made in \cite{RefWorks:81} after Lemma 3.4. 
\\
\begin{lemma}\label{e}
Let  $p>n+1$. Suppose $(x,T)$,$(y,\tau) \in \Omega_{ T}$ and $ \mathcal{R}= \lbrace( Q,s)\in M\times(0,\tau):2(|x-y|+|T-\tau|^\frac{1}{2})<|x-Q|+|T-s|^{\frac{1}{2}} \rbrace$. Then for $ 0<a<1-\frac{n+1}{p}$ there exists $K_2>0$ depending on $p, n, \overline\Omega$, $T$ and independent of $g\in L_p(M\times(0, T))$ such that,
\begin{align*}
\int_{\R}|(W(T-s,x,Q)&-W(\tau-s,y,Q))g(Q,s)|\ d\sigma \ ds\\
& \leq  K_2\left( |x-y|+|T-\tau|^\frac{1}{2}\right)^a\parallel g\parallel_{p,M\times[0,\tau]}.
\end{align*}
\end{lemma}
\begin{proof}
Using the Theorem 3.1 in \cite{RefWorks:81}, we have
\begin{align*}
&\int_{\R}|(W(T-s,x,Q)-W(\tau-s,y,Q))g(Q,s)|\ d\sigma \ ds\\
&\leq\int_{\R}C\left(\frac{|T-\tau|^\frac{1}{2}+|x-y|}{|T-s|^\frac{1}{2}+|x-Q|}\right) (1+(T-s)^{\frac{-n}{2}})\exp \left(\frac{-|x-Q|^2}{4(T-s)}\right) |g(Q,s)|\ d\sigma \ ds\\
&\leq D_1\left(\frac{1}{2}\right)^{1-a}\int_{\R}{\left(\frac{|T-\tau|^\frac{1}{2}+|x-y|}{|T-s|^\frac{1}{2}+|x-Q|}\right)}^{a} \frac{\exp \left(\frac{-|x-Q|^2}{4(T-s)}\right)}{(T-s)^{\frac{n}{2}}} |g(Q,s)|\ d\sigma\  ds\\
&\leq \tilde D_1\int_{\R}\frac{1}{|x-Q|^a }\frac{\exp \left(\frac{-|x-Q|^2}{4(T-s)}\right)}{(T-s)^{\frac{n}{2}}} |g(Q,s)|\ d\sigma\ ds\\
\end{align*}
where $D_1= C(T^{\frac{n}{2}}+1)\ \text{and}\  \tilde D_1=D_1\left(\frac{1}{2}\right)^{1-a}{\left(|T-\tau|^\frac{1}{2}+|x-y|\right)}^a$.
By hypothesis, $n+1-(n+a)p'>0.$ Pick $0<\epsilon<(n+1)-(n+a)p'$ and set $N=\frac{n-1-\epsilon-ap'}{2}$. Then there exists $c>0$ such that $w^N\cdot\exp(-w)\leq c\cdot N$ for all $w\geq 0$. Consequently,
\begin{align*}
&\tilde D_1\left(\int_{\R}\frac{1}{|x-Q|^{ap'} }\frac{\exp \left(\frac{-p'|x-Q|^2}{4(T-s)}\right)}{(T-s)^{\frac{np'}{2}}}\ d\sigma \ ds\right)^ \frac{1}{p'}\parallel g\parallel_{p,\R}\\
&\leq \tilde D_1 \left(\int_{\R} \frac{1}{|x-Q|^{ap'} }\frac{c\cdot N}{(T-s)^{\frac{np'}{2}}\left(\frac{p'|x-Q|^2}{4(T-s)}\right)^N}\right)^{\frac{1}{p'}} \parallel g\parallel_{p,\R}\\
&\leq\tilde c \tilde D_1 \left(\int_0^\tau\int_{M}{\frac{(T-s)^\frac{n-1-\epsilon-ap'}{2}}{(T-s)^{\frac{np'}{2}}}}{\frac{1}{|x-Q|^{n-1-\epsilon}}}\ d\sigma \ ds\right)^\frac{1}{p'}\parallel g(s,Q)\parallel_{p,M\times[0,\tau]} \\
&\leq\tilde c \tilde D_1 \left(\int_0^\tau{(T-s)^\frac{n-1-\epsilon-ap'-np'}{2}}\ ds\cdot\int_{M}{\frac{1}{|x-Q|^{n-1-\epsilon}}} \ \ d\sigma\right)^{\frac{1}{p'}} \parallel g\parallel_{p,M\times[0,\tau]} 
\end{align*}
Then by change of variable, there exists $C, \alpha>0$ such that
\begin{align*}
&\tilde D_1 \left(\int_0^\tau{(T-s)^\frac{n-1-\epsilon-ap'-np'}{2}}\ ds\cdot\int_{M}{\frac{1}{|x-Q|^{n-1-\epsilon}}} \ \ d\sigma\right)^{\frac{1}{p'}} \parallel g\parallel_{p,M\times[0,\tau]}\\
&\leq C\tilde D_1 \left({(T)^{\frac{n-1-\epsilon-ap'-np'}{2}+1}}\cdot\int_0^{\alpha}{\frac{1}{r^{1-\epsilon}}} dr\right)^{\frac{1}{p'}} \parallel g\parallel_{p,M\times[0,\tau]} 
\end{align*} 
The result follows.
\end{proof}\\

\begin{lemma}\label{c}
Let $p>n+1$, and suppose $(x,T)$,$(y,\tau)\in \Omega_{T}$. Then for $0<a<\frac{1}{2}-\frac{n+1}{2p}$ there exists $K_3>0$, depending on $p, n, \overline\Omega$ and $T$, and independent of  $g\in L_p(M\times(0, T))$ such that,
\begin{align*}
&\int_\tau^T\int_{M}|W(T-s,x,Q) g(Q,s)|\ d\sigma\  ds\leq K_3 (T-\tau)^a \parallel g\parallel_{p,M\times[\tau,T]}
\end{align*}
\end{lemma}
\begin{proof}By hypothesis $p>n+1$. Pick $0<\epsilon<n+1-np'$ and set $N=\frac{n-1-\epsilon}{2}$. Then there exists $c>0$ such that $w^N\cdot\exp(-w)\leq c\cdot N$ for all $w\geq 0$. Consequently,
\begin{align*}
&\int_\tau^T\int_{M}|W(T-s,x,Q) g(Q,s)|\ d\sigma\  ds\\
 & \leq \int_\tau^T\int_{M}\frac{\exp \left(\frac{-|x-Q|^2}{4(T-s)}\right)}{(T-s)^{\frac{n}{2}}} |g(Q,s)|\ d\sigma \ ds\\
&\leq C_3\int_\tau^T\int_{M}{\frac{\tilde C(T-s)^\frac{n-1-\epsilon}{2}}{(T-s)^{\frac{n}{2}}}}\cdot{\frac{1}{|x-Q|^{n-1-\epsilon}}} |g(Q,s)|\ d\sigma \ ds\\
&\leq C_3  \left(\int_\tau^T{(T-s)^\frac{n-1-\epsilon-np'}{2}}ds\cdot\int_{M}{\frac{1}{|x-Q|^{n-1-\epsilon}}} d\sigma\right)^{\frac{1}{p'}} \parallel g\parallel_{p,M\times[\tau,T]}
\end{align*}
Similarly, by change of variable there exist $\tilde C_3, \alpha>0$ such that
\begin{align*}
&C_3  \left(\int_\tau^T{(T-s)^\frac{n-1-\epsilon-np'}{2}}ds\cdot\int_{M}{\frac{1}{|x-Q|^{n-1-\epsilon}}} d\sigma\right)^{\frac{1}{p'}} \parallel g\parallel_{p,M\times[\tau,T]}\\
&\leq\tilde C_3  \left(|{(T-\tau)^{\frac{n-1-\epsilon-np'}{2}+1}}|\cdot\int_0^{\alpha}{\frac{1}{r^{1-\epsilon}}} dr\right)^{\frac{1}{p'}} \parallel g\parallel_{p,M\times[\tau,T]}\\
&\leq K_3 (T-\tau)^{\frac{n+1-\epsilon-np'}{2p'}} \parallel g\parallel_{p,M\times[\tau,T]}
\end{align*}
where $K_3>0$, depends on $p, n, \overline\Omega$ and $T$, and independent of  $g\in L_p(M\times(0, T))$.  The result follows since $0<\epsilon<n+1-np'$ is arbitrary, and $\frac{n+1-np'}{2p'}=\frac{1}{2}-\frac{n+1}{2p}$.
\end{proof}\\

\begin{proposition}\label{prop1}
Suppose $\gamma\in L_p(M\times (0, T))$ for $p>n+1$. Then the classical solution of  $(\ref{sy3})$ is H\"{o}lder continuous on $\overline{\Omega}\times(0, \hat  T)$ with H\"{o}lder exponent $0<a<1-\frac{n+1}{p}$, and there exists $\tilde K_p>0$, depending on $p, n, \overline\Omega$ and $T$, and independent of $\gamma$ such that \[ |\varphi(x,T)-\varphi(y,\tau)|\leq \tilde K_p{\left(|T-\tau|^\frac{1}{2}+|x-y|\right)}^a \parallel \gamma \parallel_{p,M\times(0, T)} \] \text{ for all $(x,T), (y,\tau)\in \Omega_{ T}$}.
\end{proposition}\\

\begin{proof}We prove this proposition for $ d=1$. The extension to arbitrary $d>0$ follows from a simple change of variables. Let $\tilde\Omega$ be an open subset of $\Omega$ with smooth boundary such that the closure of $\tilde\Omega$ is contained in $\Omega$. It is straightforward matter to apply cut-off functions and Theorem 9.1 in \cite{RefWorks:65} to obtain an estimate for $\varphi$ in $W^{2,1}_p(\tilde\Omega\times(0,T))$. Moreover, there exists $L_{p,\tilde\Omega,T}$ independent of $\gamma$ such that \[ \Vert\varphi\Vert^{(2)}_{p,\tilde\Omega_T}\leq L_{p,\tilde\Omega,T}\Vert\gamma\Vert_{p,M_T}\] Since $p>n+1$, $W^{2,1}_p(\tilde \Omega\times (0,T))$ embeds continuously into the space of H\"{o}lder continuous functions (see \cite{RefWorks:65}). As a  result  we have H\"{o}lder continuity of the solution to ($\ref{sy3})$ away from $M_T$. We want to extend this behavior to points near $M_T$.

Pick points $(x,T)$, $(y,\tau) \in \Omega_{ T}$. We know from Fabes and Riviere \cite{RefWorks:105} that the solution of ($\ref{sy3})$ is given by \[\varphi(x,T)= \int_0^T\int_M W(T-s,x,Q)g(Q,s)\ d\sigma \ ds\] where $W(T-s,x,Q)=\frac{\exp\left(\frac{-|x-Q|^2}{4(T-s)}\right)}{(T-s)^\frac{n}{2}}$, $g(Q,t)= [I +J]^{-1}\gamma(Q,t)$ and\[J(g)(Q,t)= \lim_{\epsilon\rightarrow 0^{+}} \int_0^{t-\epsilon}\int_{M}\frac{\langle y-Q,\eta_Q\rangle}{(t-s)^{\frac{n}{2}+1}}\exp \left(-\frac{\vert Q-y\vert^2}{4(t-s)}\right)g(s,y)\ d\sigma \ ds\] for almost every $ Q\in{M}$ (for smooth manifold it is true for all Q), $\eta_Q$ being the unit outward normal to $M$ at $Q.$ \begin{align*}
|\varphi(x,T)-\varphi(y,\tau)|&=|\int_0^\tau\int_{M} (W(T-s,x,Q)-W(\tau-s,y,Q))g(Q,s)\ d\sigma \ ds\\
&\quad\quad+\int_\tau^T\int_{M} W(T-s,x,Q)g(Q,s)\ d\sigma \ ds|\\
&\leq |\int_{\R^c}(W(T-s,x,Q)-W(\tau-s,y,Q))g(Q,s)\ d\sigma \ ds|\\
&\quad + |\int_{\R}(W(T-s,x,Q)-W(\tau-s,y,Q))g(Q,s)\ d\sigma \ ds|\\
&\quad+\int_\tau^T\int_M C (1+(T-s)^{\frac{-n}{2}})\exp \left(\frac{-|x-Q|^2}{4(T-s)}\right) |g(Q,s)|\ d\sigma \ ds
\end{align*}
Where $\R$ and $\R^c$ are given in Lemmas $\ref{e}$ and $\ref{c}$. Now using Lemma $\ref{a}$, Lemmas $\ref{e}$ and $\ref{c}$ for $ 0<a<1-\frac{n+1}{p}$, there exists $K_1, K_2, K_3>0$ depending on $p, n, \overline\Omega$, $T$ and independent of $g\in L_p(M\times(0, T))$, such that
\begin{align*}
|\varphi(x,T)-\varphi(y,\tau)|&\leq K_1\left( |x-y|+|T-\tau|^\frac{1}{2}\right)^a\parallel g\parallel_{p,M\times(0,\tau)}\\
 &\quad + K_2{\left(|T-\tau|^\frac{1}{2}+|x-y|\right)}^a \parallel g\parallel_{p,M\times(0,\tau)}\\
&\quad+ K_3 (T-\tau)^{\frac{n+1-\epsilon-np'}{2p'}} \parallel g\parallel_{p,M\times(\tau,T)}
\end{align*}
So, \begin{align*}
|\varphi(x,T)-\varphi(y,\tau)|\leq \tilde K_p{\left(|T-\tau|^\frac{1}{2}+|x-y|\right)}^a \parallel g\parallel_{p,M\times(0,T)}
\end{align*}
\end{proof}\\

Now we combine H\"{o}lder estimates and Theorem 9.1 in chapter 4 of \cite{RefWorks:65} to get the existence of a H\"{o}lder continuous solution to system $(\ref{m2})$ for any finite time $T>0$. \\

{\bf Proof of Theorem $\ref{n}$:} Chapter 4, Theorem 5.1 in \cite{RefWorks:65}  implies ($\ref{m2}$) has the unique weak solution. In order to get H\"{o}lder estimates, we break  $(\ref{m2})$ into two sub systems. To  this end, consider

\begin{align}
{\varphi_2}_t\nonumber&=  d\Delta \varphi_2+\theta & x\in \Omega,\quad & 0<t<T\\ 
d\frac{\partial \varphi_2}{\partial \eta}&=d\frac{\partial\varphi_0}{\partial\eta}& x\in M,\quad & 0<t< T\\ \nonumber
\varphi_2&=\varphi_0& x\in\Omega ,\quad &t=0
\end{align}
\begin{align}
{\varphi_1}_t\nonumber&=  d\Delta \varphi_1&
 x\in \Omega,\quad &0<t<T\\ 
d\frac{\partial \varphi_1}{\partial \eta}&=\gamma - d\frac{\partial\varphi_0}{\partial\eta}& x\in M,\quad &0<t< T\\ \nonumber
\varphi_1&=0& x\in\Omega ,\quad &t=0
\end{align}
 From Lemma $\ref{L1.5}$ there exists a unique solution of $(5.1)$ in $W^{2,1}_p(\Omega\times(0, T))$, and a constant $C_1( T, p)>0$ independent of $\theta$ and $\varphi_0$ such that 
 \[ \Vert \varphi_2 \Vert^{(2)}_{p,\Omega\times(0, T)} \leq C_1( T, p) (\Vert \theta\Vert_{p,\Omega\times(0,T)}+{\Vert \frac{\partial\varphi_0}{\partial\eta}\Vert}_{p,(\partial\Omega\times(0, T))}^{(1-\frac{1}{p},\frac{1}{2}-\frac{1}{2p})})+\Vert \varphi_0\Vert^{(2)}_{p,\Omega}\]
Using proposition $\ref{prop1}$, there exists $C_2(T,0)>0$ independent of $\gamma$ and $\varphi_0$ so that the unique weak solution to (5.2) satisfies,
\[\vert \varphi_1 \vert^{(\beta)}_{\Omega\times(0,T)} \leq C_2( T, p)\left[ \Vert \gamma\Vert_{p,M\times(0,T)}+ \Vert \frac{\partial\varphi_0}{\partial\eta}\Vert_{p,M\times(0, T)}\right] \]
where $0<\beta<1-\frac{n+1}{p}$. By linearity,  $\varphi=\varphi_1+\varphi_2$ solves $(\ref{m2})$. Moreover, for $p>n+1$, $W^{2,1}_p(\Omega\times(0,T))$ embeds continuously into $C^{\beta, \frac{\beta}{2}}(\overline\Omega_T)$. So, there exists $C( T, p)>0$ independent of $\theta$, $\gamma$ and $\varphi_0$ such that 
\begin{align}
\vert \varphi\vert^{(\beta)}_{\Omega\times(0, T)} &\leq C( T, p) (\Vert \theta\Vert_{p,\Omega\times(0, T)}+\Vert \gamma\Vert_{p,M\times(0,T)}+\Vert \varphi_0\Vert^{(2)}_{p,\Omega})
\end{align}\\

\begin{rems}
We will use these H\"{o}lder estimates to obtain sup norm estimates, and local existence results for $(\ref{sy5})$.
\end{rems}\\

\section{Proof of Theorems $\ref{lo}$ and $\ref{great}$}
\setcounter{equation}{0}
\subsection{Local Existence}
\quad \\
\begin{theorem}\label{glo}
Suppose $F, G$ and $H$ are Lipschitz. Then $(\ref{sy5})$ has a unique global solution.
\end{theorem}\\

\begin{proof}
 Let $T>0$, Fix $(u_0,v_0)\in W_p^{2}(\Omega)\times W_p^{2}(\Omega)$ such that they satisfy the compatibility condition
\begin{align}\label{comp}
 D{\frac{ \partial {u_0}}{\partial \eta}} =G( u_0, v_0)\quad \text{on $M$}.
\end{align}Set \begin{align*}
X=\lbrace ( u, v)\in  C(\overline\Omega\times[0,T])\times C(M\times[0,T]):\ &  u(x,0)=0  ,\forall\ x\in\overline\Omega, v(x,0)=0 ,\forall\ x\in M\rbrace
\end{align*}
Note $(X, \Vert \cdot\Vert_{\infty})$ is a Banach space. Let $(u,v) \in X$. Now consider 
\begin{align}\label{fix}
 U_t\nonumber&= D\Delta U+H( u+u_0)
& x\in \Omega,\quad & 0<t<T
\\\nonumber V_t&=\tilde D\Delta_{M} V+F( u+u_0, v+v_0)& x\in M,\quad &0<t<T\\
 D\frac{\partial U}{\partial \eta}&=G( u+u_0, v+v_0)& x\in M,\quad &0<t<T\\\nonumber
U&=u_0  &x\in\Omega ,\quad &t=0 \\\nonumber
V&=v_0& x\in M ,\quad& t=0\end{align}
 From Theorems $\ref{3}$ and $\ref{n}$, ($\ref{fix}$) possesses a unique weak  solution $(U,V)\in V_2^{1,\frac{1}{2}}(\Omega_T)\times W_p^{2,1}(M_T)$. Furthermore, from embeddings, $(U,V)\in  C(\overline\Omega\times[0,T])\times C(M\times[0,T])$. Define \[S:X\rightarrow X\ \text{via}\ S(u,v)=(U-u_0,V-v_0),\]where $(U,V)$ solves  ($\ref{fix}$). We will see that $S$ is continuous and compact. Let $(u, v)$, $(\tilde u, \tilde v) \in X$. Then \[S(u,v)-S(\tilde u,\tilde v)=(U-\tilde U, V-\tilde V)\] Using  linearity, $(U-\tilde  U, V-\tilde V)$ solves
\begin{align} {U}_t-{\tilde U}_t&= \nonumber D\Delta (U-\tilde U)+H( u+u_0)-H( \tilde u+u_0)
 &x\in \Omega,\quad &0<t<T
\\ \nonumber {V}_t-{\tilde V}_t&=\tilde D\Delta_{M}(V-\tilde V)+F(u+u_0, v+v_0)-F( \tilde u+u_0, \tilde v +v_0)& x\in M,\quad& 0<t<T\\  \nonumber
 D\frac{\partial (U-\tilde U)}{\partial \eta}&=G( u+u_0, v+v_0)-G(\tilde u+u_0, \tilde v+v_0)&x\in M,\quad &0<t<T\\ \nonumber
U-\tilde U&=0 & x\in\Omega ,\quad& t=0\\ \nonumber V-\tilde V&=0 & x\in M ,\quad &t=0\end{align} 
From Theorem $\ref{n}$, if $p>n+1$ there exists $K$ independent of $H, G, F, u,v, \tilde u, \tilde v$ such that
\begin{align*}
\Vert U-\tilde U\Vert_{\infty,\Omega_T}+\Vert V-\tilde V\Vert_{\infty, M_T}\leq K&\left (\Vert F(u+u_0, v+v_0)-F(\tilde u+u_0,\tilde v+v_0)\Vert _{p,M_T}\right.\\&\left. +\Vert G(u+u_0, v+v_0)-G(\tilde u+u_0,\tilde v+v_0)\Vert_{p,M_T}\right. \\
&\left.+\Vert H(u+u_0)-H(\tilde u+u_0)\Vert _{p,\Omega_T}\right)
\end{align*}
Using the boundedness of $\Omega$ and $M$, there exists $\tilde K>0$ such that
\begin{align*}
\Vert U-\tilde U\Vert_{\infty,\Omega_T}+\Vert V-\tilde V\Vert_{\infty, M_T}\leq \tilde K&\left (\Vert F(u+u_0, v+v_0)-F(\tilde u+u_0,\tilde v+v_0)\Vert _{\infty,M_T}\right.\\&\left. +\Vert G(u+u_0, v+v_0)-G(\tilde u+u_0,\tilde v+v_0)\Vert_{\infty,M_T}\right. \\
&\left.+\Vert H(u+u_0)-H(\tilde u+u_0)\Vert _{\infty,\Omega_T}\right)
\end{align*}
Since, $F,G,H$ are Lipschitz functions there exists $\tilde M>0$ such that \begin{align*}
\Vert U-\tilde U\Vert_{\infty,\Omega_T}+\Vert V-\tilde V\Vert_{\infty,M_T}
&\leq \tilde M(\Vert u-\tilde u\Vert_{\infty, \overline\Omega_T}+\Vert v-\tilde v)\Vert _{\infty,M_T})
\end{align*}
Therefore $S$ is continuous with respect to the sup norm.
Moreover, for $p>n+1$, from Theorem $\ref{3}$, $\ref{n}$, and Lemma $\ref{L3}$, there exists $\hat C(T, p)>0$, independent of $F(u+u_0,v+v_0), G(u+u_0,v+v_0), H(u+u_0), u_0$ and $v_0$ such that for all $0<\alpha<1-\frac{n}{p}$, $0<\beta<1-\frac{n+1}{p}$,
\begin{align}\label{precompact}
\vert U\vert^{(\beta)}_{\Omega_T}+\vert V\vert^{(\alpha)}_{M_T}&\leq \hat C(T,p) (\Vert H(u+u_0)\Vert_{p,\Omega_T}+\Vert G(u+u_0,v+v_0)\Vert_{p,M_T}\\ &\nonumber\quad+ \Vert F(u+u_0,v+v_0)\Vert_{p,M_T}+\Vert v_0\Vert^{(2)}_{p,M}+\Vert u_0\Vert^{(2)}_{p,\Omega})
\end{align}
Using $(\ref{precompact})$, $S$ maps bounded sets in $X$ to precompact sets, and hence $S$ is compact with respect to the sup norm. Now we show $S$ has a fixed point. To this end, we show that the set A=$\lbrace ( u, v)\in X : (u,v)=\lambda S(u, v) \ \text{for some}\ 0<\lambda\leq1\rbrace$ is bounded in $X$ with respect to the sup norm. Let $( u,v)\in$ A. Then there exists $0<\lambda\leq 1$ such that $(\frac{u}{\lambda},\frac{v}{\lambda})= S( u, v)$. Therefore if $(\hat u,\hat v)=(u+\lambda u_0, v+\lambda v_0)$ then 
\begin{align*}
\hat u_t\nonumber&=  D\Delta \hat u+\lambda H( u+u_0)
& x\in \Omega,\quad & 0<t<T
\\\nonumber  \hat v_t&=\tilde D\Delta_{M} \hat v+\lambda F(u+u_0, v+v_0)& x\in M,\quad &0<t<T\\
 D\frac{\partial \hat u}{\partial \eta}&= \lambda G(u+u_0, v+v_0)& x\in M,\quad &0<t<T\\\nonumber
\hat  u&= \lambda u_0  &x\in\Omega ,\quad &t=0 \\\nonumber
\hat  v&=\lambda v_0& x\in M ,\quad& t=0
\end{align*}
From Theorem $\ref{n}$ and $H, F$ and $G$ being Lipschitz, there exists $N>0$ such that $\Vert(\hat u,\hat v)\Vert_\infty\leq N$, with $N$ independent of $\lambda, u$ and $v$. Since $\Vert( u, v)\Vert_\infty\leq \Vert(\hat u,\hat v)\Vert_\infty\leq N$, hence boundedness of the set is accomplished. Thus, applying Schaefer's theorem (see \cite{RefWorks:52}), we
conclude $S$ has a fixed point ($U,V$). Further, $(U+u_0, V+v_0)$ is a solution of ($\ref{sy5}$). Moreover, bootstrapping the regularity of this solution using well known estimates, we obtained a {\it solution} to ($\ref{sy5}$) according to Definition $\ref{blah}$.

Finally, we show the solution of ($\ref{sy5}$) is unique. Suppose $(u,v),(\hat u,\hat v)$ solve ($\ref{sy5}$). Then, $(u-\hat u,v-\hat v)$ satisfies
\begin{align}{u}_t-{\hat u}_t&= \nonumber D\Delta (u-\hat u)+H( u)-H( \hat u)
& x\in \Omega,\quad& t>0
\\ \nonumber {v}_t-{\hat v}_t&=\tilde D\Delta_{M}( v-\hat v)+F( u,v)-F( \hat u, \hat v)& x\in M,\quad& t>0\\\nonumber   D\frac{\partial (u-\hat u)}{\partial \eta}&=G( u, v)-G( \hat u,\hat  v)& x\in M,\quad & t>0\\ \nonumber
u-\hat u&=0& x\in\Omega ,\quad &t=0\\ \nonumber v-\hat v&=0& x\in M,\quad &t=0\end{align} Taking the dot product of the ${v}_t-{\hat v}_t$ equation with $(v-\hat v)$, and the ${u}_t-{\hat u}_t$ equation with $(u-\hat u)$, and integrating over $M$ and $\Omega$ respectively, yields
\begin{align*}
\frac{1 }{2 }\frac{d}{dt}(\Vert v-\hat v\Vert_{2,M}^2 &+ \Vert u-\hat u\Vert _{2,\Omega}^2) + D\Vert \nabla (u-\hat u)\Vert^2_{2,\Omega}\\
&\leq \Vert v-\hat v\Vert_{2,M}\Vert F(u,v)-F(\hat u,\hat v)\Vert_{2,M}+\Vert u-\hat u\Vert_{2,\Omega}\Vert H(u)-H(\hat u)\Vert_{2,\Omega}\\
 &\quad +\Vert u-\hat u\Vert_{2, M}\Vert G(u,v)-G(\hat u,\hat v)\Vert_{2,M}\\
&\leq K\Vert v-\hat v\Vert_{2,M} \left(\Vert u-\hat u\Vert_{2,M}+\Vert v-\hat v\Vert_{2,M}\right)\\
&\quad +K\Vert u-\hat u\Vert _{2, M} \left(\Vert u-\hat u\Vert_{2,M}+\Vert v-\hat v\Vert_{2,M}\right)++K\Vert u-\hat u\Vert_{2 ,\Omega}^2\\
&\leq  K(\Vert v-\hat v\Vert^2_{2,M}+\Vert u-\hat u\Vert_{2 ,M}^2)\\
&\quad+2K\Vert u-\hat u\Vert_{2,M}\Vert v-\hat v\Vert_{2,M}+K\Vert u-\hat u\Vert_{2 ,\Omega}^2\\
&\leq 2 K(\Vert v-\hat v\Vert^2_{2,M}+\Vert u-\hat u\Vert_{2 ,M}^2)+K\Vert u-\hat u\Vert_{2 ,\Omega}^2
\end{align*}
From Lemma $\ref{i}$, for $p=2$ and $\epsilon=\frac{d_{min}}{2K}=\frac{\min\lbrace d_j:1\leq j\leq k\rbrace}{2K}$, we have
\begin{eqnarray}\label{b}
\Vert u-\hat u\Vert_{2,M}^2\leq \frac{d_{min}}{2K} \Vert\nabla( u-\hat u)\Vert_{2,\Omega}^2+\tilde C_{\epsilon}\Vert u-\hat u\Vert_{2,\Omega}^2
\end{eqnarray}
Using $(\ref{b})$  
\begin{align*}
\frac{1 }{2 }\frac{d}{dt}\left(\Vert v-\hat v\Vert_{2,M}^2 + \Vert u-\hat u\Vert _{2,\Omega}^2\right)
&\leq 2K\Vert v-\hat v\Vert^2_{2,M}+K(1+2\tilde C_\epsilon)\Vert u-\hat u\Vert_{2,\Omega}^2\\
&\leq C_{\epsilon,k} \left(\Vert v-\hat v\Vert^2_{2,M}+\Vert u-\hat u\Vert_{2 ,\Omega}^2\right)
\end{align*}
Observe, $(u-\hat u)=(v-\hat v)=0$ at $ t=0$ and $\left(\Vert u-\hat u\Vert^2_{2,\Omega}+\Vert v-\hat v\Vert_{2,M}^2\right)\geq 0$. Therefore, applying Gronwall's inequality, $v=\hat v$ and $u=\hat u$. Hence system ($\ref{sy5}$) has the unique global solution.
\end{proof}\\

{\bf Proof of Theorem $\ref{lo}$:} Recall that $u_0 \in W_p^{2}(\Omega)$ and $v_0\in W_p^{2}(M)$ with $p>n$,  and $ u_0,v_0$ satisfies the compatibility condition for $p>3$. From Sobolev imbedding (see \cite{RefWorks:53}, \cite{RefWorks:65}), $u_0,v_0$ are bounded functions. Therefore there exists $ \tilde r>0$ such that $\Vert u_0(\cdot)\Vert_{\infty,\Omega} \leq \tilde r$, $\Vert v_0(\cdot)\Vert_{\infty,M}\leq \tilde r $.

 For each $r>\tilde r$, we define cut off functions $\phi_r\in C_{0}^{\infty}({\mathbb{R}}^{k},[0,1])$ and $\psi_{r}\in C_{0}^{\infty}(({\mathbb{R}}^{k}\times{\mathbb{R}}^{m}),[0,1])$  such that $ \phi_r(z)=1$ for all $\vert z\vert\leq r$, and  $ \phi_r(z)=0$ for all $\vert z\vert> 2r$. Similarly $\psi_{r}(z,w)=1$ when $\vert z\vert\leq r$ and $\vert w\vert\leq r$, and $\psi_{r}(z,w)=0$ when $\vert z\vert>2r,$ or $\vert w\vert>2r$. In addition,  we define $ H_r= H\phi_r,
F_{r}= F\psi_{r} $ and
$ G_{r}=  G\psi_r$.
From construction, $H_r(z)= H(z), F_{r}(z,w)= F(z,w)$ and $G_{r}(z,w)= G(z,w)$
when $\vert z\vert\leq r$ and $\vert w\vert\leq r$.
Also, there exists $M_r>0$ such that $H_r, G_r$ and $F_r$ are Lipschitz functions with Lipschitz coefficient $M_r$. 
Consider the ``restricted'' system

\begin{align}\label{ys}
u_t &=  D\Delta u+H_r(u)& x\in \Omega,\quad & t>0 \nonumber\\
\nonumber v_t&= \tilde D\Delta_{M} v+F_r(u,v)& x\in M,\quad &t>0 \\ 
 D\frac{\partial u}{\partial \eta} &=G_r(u,v)& x\in M, \quad &t>0\\
\nonumber u&=u_0  & x\in\Omega ,\quad &t=0\\ \nonumber v&=v_0 &  x\in M ,\quad & t=0
\end{align}

From Theorem $\ref{glo}$, $(\ref{ys})$ has a unique global solution $(u_r,v_r)$. If $\Vert u(\cdot,t)\Vert_{\infty,\Omega},\Vert v(\cdot,t)\Vert_{\infty, M}\leq r$ for all $t\geq 0$, then $(u_r,v_r)$ is a global solution to ($\ref{sy5}$). If not, there exists $T_r>0$ such that  \[\Vert u_r(\cdot,t)\Vert_{\infty,\Omega}+\Vert v_r(\cdot,t)\Vert_{\infty, M}\leq r\quad\forall t\in[0, T_r]\] and for all $\tau>T_r$ there exists $t$ such that $T_r<t<\tau$, and $x\in\overline\Omega$ and $z\in M$, such that \[\vert u_r(x,t)\vert+\vert v_r(z,t)\vert> r \] Note that $T_r$ is increasing with respect to $r$. Let $T_{\max}=\lim_{r\rightarrow\infty}T_r$. Now we define $(u,v)$ as follows. Given $0<t<T_{\max}$, there exists $r>0$ such that $  t<T_r\leq T_{\max}$. For all $x\in\overline\Omega$, $ u(x,t)=u_r(x,t)$, and for all $x\in M$, $v(x,t)=v_r(x,t)$.  Furthermore $(u,v)$ solves $(\ref{sy5})$ with $T=T_{\max}$. Also, uniqueness of $(u_r,v_r)$ implies uniqueness of $(u,v)$.
It remains to show that the solution of  ($\ref{sy5}$) is maximal and if $T_{\max}<\infty$ then \[\displaystyle \limsup_{t \to T^-_{\max}}\Vert u(\cdot,t)\Vert_{\infty,\Omega}+\displaystyle \limsup_{t \to T^-_{\max} }\Vert v(\cdot,t)\Vert_{\infty,M}=\infty.\]
 Suppose $T_{\max}<\infty$ and set, \[\displaystyle \limsup_{t \to T^-_{\max}}\Vert u(\cdot,t)\Vert_{\infty,\Omega}+\displaystyle \limsup_{t \to T^-_{\max} }\Vert v(\cdot,t)\Vert_{\infty,M}=R.\] If $R=\infty$ then $(u,v)$ is a maximal solution. If $R<\infty$ there exists $L>0$ such that  \[\Vert u\Vert_{\infty,\Omega\times(0,T_{\max})}+\Vert v\Vert_{\infty,M \times(0,T_{\max})}\leq L.\] As a result,  $T_{2L}>T_{\max}$, contradicting the construction of $T_{2L}$.\quad$\square$\\
\\
Now we prove that under some extra assumptions that the solution to $\left (\ref{sy5}\right)$ is componentwise nonnegative. Consider the system
\begin{align}\label{+}
 u_t\nonumber&= D\Delta u+H(u^{+})
 & x\in \Omega, \quad&0<t<T
\\\nonumber v_t&=\tilde D\Delta_{M} v+F(u^{+},v^{+})& x\in M,\quad& 0<t<T\\  D\frac{\partial u}{\partial \eta}&=G(u^{+},v^{+}) & x\in M, \quad&0<t<T\\\nonumber
u&=u_0  &x\in\Omega ,\quad& t=0\\\nonumber v&=v_0 & x\in M ,\quad &t=0\end{align}
where $u^{+} = \max (u,0)$ and $u^{-}=-\min (u,0)$.\\

\begin{proposition}\label{p}
Suppose $F, G$ and $H$ are locally Lipschitz, quasi positive functions, and $u_0, v_0$ are componentwise nonnegative functions. Then $(\ref{+})$ has a unique componentwise nonnegative solution.
\end{proposition}\\

\begin{proof}
 Note that $F(u^{+},v^{+})$, $G(u^{+},v^{+})$ and $H(u^{+})$ are locally Lipschitz functions of $u$ and $v$. Therefore  from Theorem $\ref{lo}$ there exists a unique maximal solution to ($\ref{+}$) on $(0, T_{\max})$. Consider $\left (\ref{+}\right)$ componentwise. Multiply the ${v_i}_t$ equation by $v_i^{-}$ and the ${u_j}_t$ equation by $u_j^{-}$, 
\begin{eqnarray}
v_i^{-}\frac{\partial v_i}{\partial t}&= \tilde d_i v_i^{-}\Delta_M v_i + v_i^{-} F_i(u^{+}, v^{+})\label{vu}\\
u_j^{-}\frac{\partial u_j}{\partial t} &=  d_j u_j^{-}\Delta u_j+ u_j^{-} H_j(u^{+})\label{uv}
\end{eqnarray}
Since $w^{-}\frac{dw}{dt}=\frac{-1}{2}\frac{d}{dt}(w^{-})^2$, 
\begin{align*}
\frac{1}{2}\frac{\partial}{\partial t}(v_i^{-})^2+\frac{1}{2}\frac{\partial}{\partial t}(u_j^{-})^2 &= -\tilde d_i v_i^{-}\Delta_M v_i -v_i^{-} F_i(u^{+}, v^{+})\\&\quad- d_j u_j^{-}\Delta u_j-u_j^{-} H_j(u^{+})
\end{align*}
Integrating ($\ref{vu}$) and ($\ref{uv}$) over $M$ and $\Omega$ respectively, gives
\begin{align*}
\frac{1}{2}\frac{d}{dt}\Vert v_i^{-}(\cdot,t)\Vert^2_{2,M}&+\frac{1}{2}\frac{d}{dt}\Vert u_j^{-}(\cdot,t)\Vert^2_{2,\Omega}+\tilde d_i \int_M |\nabla v_i^{-}|^2\ d\sigma+ d_j \int_\Omega |\nabla u_j^{-}|^2\ dx \\&= -\int_\Omega u_j^{-} H_j(u^{+})\ dx-\int_M u_j^{-} G_j(u^{+},v^{+})\ d\sigma -\int_Mv_i^{-} F_i(u^{+}, v^{+})\ d\sigma
\end{align*}
 Since $F, G$  and $H$ are quasi-positive and $\tilde d_i , d_j> 0$,
\begin{align*}
\frac{1}{2}\frac{d}{dt}\Vert v_i^{-}(\cdot,t)\Vert^2_{2,M}&+\frac{1}{2}\frac{d}{dt}\parallel u_j^{-}(\cdot,t)\parallel^2_{2,\Omega}\leq 0
\end{align*}
Therefore, the solution $(u,v)$ is componentwise nonnegative.
\end{proof}\\
\begin{corollary}\label{needco}
Suppose $F, G$ and $H$ are locally Lipschitz, quasi positive functions, and $u_0, v_0$ are componentwise nonnegative functions. Then the unique solution $(u,v)$ of $(\ref{sy5})$  is componentwise nonnegative. 
\end{corollary}\\
 
\begin{proof}
From Theorem $\ref{lo}$ and Proposition $\ref{p}$, there exists a unique, componentwise nonnegative and maximal solution $(u,v)$ to ($\ref{+}$). In fact $(u,v)$ also solves $(\ref{sy5})$. The result follows.
\end{proof}

\subsection{Bootstrapping Strategy}

The following system will play a central role in duality arguments.\begin{align}\label{aj2}
&\Psi_t =-\tilde d\Delta_M \Psi-\tilde\vartheta & (x,t)\in M\times (\tau,T)\nonumber\\ 
&\Psi = 0 & x\in M , t=T\tag{6.9a}\nonumber
\end{align}
\begin{align}\label{ajj3}
\varphi_t &=- d\Delta \varphi-\vartheta & (x,t)\in \Omega\times (\tau,T)\nonumber\\ \kappa_1d\frac{\partial \varphi}{\partial \eta}&+\kappa_2\varphi=\Psi & (x,t)\in M\times (\tau,T)\tag{6.9b}\nonumber\\
\varphi&=0 &x\in\Omega ,\quad t=T\nonumber\end{align}
Here, $p>1$, $0<\tau<T$, $\tilde\vartheta\in L_{p}{(M\times(\tau,T))}$ and $\tilde\vartheta\geq 0$, and  $\vartheta\in L_{p}{(\Omega\times(\tau,T))}$ and $\vartheta\geq 0$. Also $d>0$, $\tilde d>0$, and $\kappa_1,\kappa_2 \in\mathbb{R}$ such that $\kappa_1\geq 0$ and $\kappa_1\kappa_2\neq 0$. Lemmas $\ref{manifold}$ to $\ref{lp5}$ provide helpful estimates.\\

\begin{lemma}\label{manifold}
$(\ref{aj2})$ has a unique nonnegative solution $\Psi\in{W_{p}}^{2,1}{(M\times(\tau,T))}$ and there exists $C_{p,T}>0$ independent of $\tilde\vartheta$ such that \[\Vert \Psi\Vert_{p,M\times(\tau,T)}^{(2)}\leq C_{p,T}\Vert\tilde\theta\Vert_{p,M\times(\tau,T)}\]
\end{lemma}\\
\begin{proof}
The result follows from Theorem $\ref{3}$ and the comparison principle.
\end{proof}\\
\begin{lemma}\label{Lp3}
Let p $>1$, $\kappa_1 \geq 0$ and if $\kappa_1=0$ then $\kappa_2>0$. Suppose $\Psi$ is the unique nonnegative solution of $(\ref{aj2})$. Then $(\ref{ajj3})$ has a unique nonnegative solution $\varphi\in W_{p}^{2,1}{(\Omega\times(\tau,T))}$. Moreover, there exists $C_{p,T}>0$ independent of $\vartheta$ and $\tilde\vartheta$ and dependent on $d,\tilde d,\kappa_1$ and $\kappa_2$ such that
\[{\Vert \varphi\Vert}_{p,(\Omega\times(\tau,T))}^{(2)}\leq C_{p,T}(\Vert\tilde\theta\Vert_{p,M\times(\tau,T)}+\Vert\theta\Vert_{p,\Omega\times(\tau,T)})\]\end{lemma}

\begin{proof}
The result follows from Lemma $\ref{manifold}$, Sobolev embedding and similar arguments of proof on page 342, section 9 of chapter 4 in \cite{RefWorks:65}, and the comparison principle.
\end{proof}\\

\begin{rems}\label{hol}
If $p>n+2$ and $\kappa_1 > 0$, then $\nabla \varphi$ is H\"older continuous in $x$ and $t.$ See the Corollary after Theorem 9.1, (page 342) chapter 4 of \cite{RefWorks:65}.
\end{rems}\\
\begin{lemma}\label{more}
Suppose $l>0$ is a non integral number, $\kappa_1> 0$, $d>0$, $\vartheta\in C^{l,\frac{l}{2}}(\overline\Omega\times[\tau,T])$, $\tilde\vartheta\in C^{l,\frac{l}{2}}(M\times[\tau,T])$, $\varphi(x,T)\in C^{2+l}(\overline\Omega)$ and $\Psi\in C^{l+1, \frac{(l+1)}{2}}(M\times[\tau,T])$. Then $(\ref{ajj3})$ has a unique solution in $C^{l+2,\frac{l}{2}+1}(\overline\Omega\times[\tau,T])$. Moreover there exists $c>0$ independent of $\Psi$ and $\vartheta$ such that  \[ |\varphi|^{(l+2)}_{\Omega\times[\tau,T]}\leq c\left(|\vartheta|^{(l)}_{\Omega\times[\tau,T]}+|\Psi|^{(l+1)}_{M\times(\tau,T)}\right)\]
\end{lemma}
\begin{proof} See Theorem 5.3 in chapter 4 of \cite{RefWorks:65}.
\end{proof}\\
\begin{lemma}\label {lp6}
Suppose $1<p<\infty$, $\kappa_1 > 0$, and $ r, s$ are positive integers. If $q\geq p$ and $2-2r-s-\left(\frac{1}{p}-\frac{1}{q}\right)(n+2)\geq 0$ then there exists $\tilde K>0$ depending on $\Omega, r, s, n, p$ such that    \[ \Vert D_t^rD_x^s \varphi\Vert_{q,\Omega\times(\tau,T)}\leq \tilde K \Vert \varphi\Vert_{p,\Omega\times(\tau,T)}^{(2)}\] for all $\varphi \in W^{2,1}_p(\Omega\times(\tau,T))$.
\end{lemma}\\
\begin{proof}
See Lemma 3.3 in chapter 2 of \cite{RefWorks:65}.
\end{proof}\\
\begin{lemma}\label{lp5}
Suppose $1<p<\infty$, $\kappa_1>0$, and $ r,s, m$ are positive integers satisfying $2r+s<2m-\frac{2}{p}$. There exists $c>0$ independent of $\varphi\in{W_{p}}^{2m,m}{(\Omega\times(\tau,T))}$ such that
\begin{center}
$D^{r}_{t}D^{s}_{x}\varphi|_{t=\tau}\in {W_{p}}^{2m-2r-s-\frac{2}{p}}(\Omega)$ and ${\Vert\varphi\parallel}^{(2m-2r-s-\frac{2}{p})}_{p,\Omega} \leq c {\Vert \varphi\parallel}^{(2m)}_{p,\Omega\times(\tau,T)}$
\end{center}
In addition, when $2r+s<2m-\frac{1}{p}$,
\begin{center}
$D^{r}_{t}D^{s}_{x}\varphi|_{M\times(\tau,T)}\in {W_{p}}^{2m-2r-s-\frac{1}{p},\   m-r-\frac{s}{2}-\frac{1}{2p}}(M\times(\tau,T))$\\
and 
 ${\Vert\varphi\parallel}^{(2m-2r-s-\frac{1}{p})}_{p,M\times(\tau,T)} \leq c {\Vert \varphi\parallel}^{(2m)}_{p,\Omega\times(\tau,T)}$
\end{center}
\end{lemma}

\begin{proof}
See Lemma 3.4 in chapter 2 of \cite{RefWorks:65}.
\end{proof}\\
\begin{lemma}\label{Ldir}
Let p $>1$, $\kappa_1=0$ and suppose $0\leq \vartheta\in L_{p}{(\Omega\times(\tau,T))}$, and $\Psi$ is a unique solution of $(\ref{aj2})$. Then $\Psi\in W_{p}^{2-\frac{1}{p},1-\frac{1}{2p}}(M\times(\tau,T))$, and $(\ref{ajj3})$ has a unique solution $\varphi\in W_{p}^{2,1}{(\Omega\times(\tau,T))}$. Moreover, there exists $C_{p,T}>0$ independent of $\vartheta$ and dependent on $d$, and $\kappa_2$ such that
\[{\Vert \varphi\Vert}_{p,\Omega\times(\tau,T)}^{(2)}\leq C_{p,T}( {\Vert \vartheta\Vert}_{p,\Omega\times(\tau,T)}+{\Vert \tilde\vartheta\Vert}_{p, M\times(\tau,T)})\]
\end{lemma}
\begin{proof}
The result follows from  Theorem 9.1 in chapter 4  of \cite{RefWorks:65}, Lemma $\ref{manifold}$, and Sobolev embedding.
\end{proof}\\
\begin{rems}
If $p>\frac{n+2}{2}$ , $\kappa_1=0$ and $\varphi$ satisfies system $(\ref{ajj3})$, then $\varphi$ is a H\"older continuous function in $x$ and $t$. See the Corollary after Theorem 9.1, chapter 4 of \cite{RefWorks:65}.
\end{rems}\\
\begin{rems}\label{holl} By Lemma $\ref{manifold}$, Lemma $\ref{Lp3}$, Lemma $\ref{lp5}$, and Sobolev embedding, we have $\varphi(\cdot,\tau)\in W^{2-\frac{2}{p}}_p(\Omega),\ \Psi(\cdot,\tau)\in W^{2-\frac{2}{p}}_p(M)$, and there exists $c>0$ independent of $\varphi$, $\Psi$ such that \[{\Vert\varphi(\cdot,\tau)\parallel}^{(2-\frac{2}{p})}_{p,\Omega} \leq c ( {\Vert \vartheta\Vert}_{p,\Omega\times(\tau,T)}+{\Vert \tilde\vartheta\Vert}_{p, M\times(\tau,T)})\] \[{\Vert\Psi(\cdot,\tau)\parallel}^{(2-\frac{2}{p})}_{p,M}\leq c {\Vert \vartheta\Vert}_{p,\Omega\times(\tau,T)}\] respectively. Moreover, if  $p>n$ there exists $c>0$ independent of $\varphi$, $\Psi$ such that \[ {\Vert \varphi\Vert}_{\infty,\Omega\times(\tau,T)}\leq c{\Vert \varphi(\cdot,\tau)\Vert}^{(2-\frac{2}{p})}_{p,\Omega}\] \[{\Vert \Psi\Vert}_{\infty,M\times(\tau,T)}\leq c{\Vert \Psi(\cdot,\tau)\Vert}^{(2-\frac{2}{p})}_{p,M}\] respectively. 
\end{rems}\\

\begin{lemma}\label{flat}
Let $1<p< n+2$ and $1<q\leq \frac{(n+1)p}{n+2-p}$. There exists a constant $\hat C>0$ depending on $p,T-\tau,M$ and $n$   such that if $\varphi\in W^{2,1}_p(\Omega\times(\tau,T))$, then \[\left\Vert\frac{\partial\varphi}{\partial\eta}\right\Vert_{q,M\times(\tau,T)}\leq \hat C {\left \Vert \varphi\right\Vert}^{(2)}_{p,\Omega\times(\tau,T)}\]
\end{lemma}
\begin{proof}
It suffices to consider the case when $\varphi$ is smooth  in $\overline\Omega\times[\tau,T]$, as such functions are dense in $W^{2,1}_p(\Omega\times(\tau,T))$.  $M$ is a $C^{2+\mu}$, $n-1$ dimensional manifold ($\mu>0$). Therefore, for every $\hat\xi\in M$ there exists $\epsilon_{\hat\xi}>0$, an open set $V\subset\mathbb{R}^n$ containing $0$, and a $C^{2+\mu}$ diffeomorphism $\psi:V\rightarrow B(\hat\xi,\epsilon_{\hat\xi})$ such that $\psi(\bf 0)=\hat\xi$, $\psi(\lbrace x\in V: x_n>0\rbrace)= B(\hat\xi,\epsilon_{\hat\xi})\cap \Omega$ and $\psi( \lbrace x\in V: x_n=0\rbrace)=B(\hat\xi,\epsilon_{\hat\xi})\cap M$. 
Since $\psi$ is a $C^2$ diffeomorphism, $(\psi^{-1})_n$, the nth component of $\psi^{-1}$, is differentiable in $B(\hat\xi,\epsilon_{\hat\xi})$, and by definition of $\psi$, $(\psi^{-1})_n(\xi)=0$ if and only if $\xi\in B(\hat\xi,\epsilon_{\hat\xi})\cap M$. Further, $\nabla (\psi^{-1})_n(\xi)$ is nonzero and orthogonal to $B(\hat\xi,\epsilon_{\hat\xi})\cap M$ at each $\xi\in B(\hat\xi,\epsilon_{\hat\xi})\cap M$. Without loss of generality, we assume the outward unit normal is given by  \[\eta(\xi)=\frac{\nabla (\psi^{-1})_n(\xi)}{|(\nabla\psi^{-1})_n(\xi)|}\quad \forall\ \xi\in B(\hat\xi,\epsilon_{\hat\xi})\cap M\]
 We know,\[
\frac{\partial\varphi}{\partial\eta}(\xi,t)=\nabla_{\xi}\varphi(\xi,t)\cdot\eta(\xi)\quad\forall\ (\xi,t)\in B(\hat\xi,\epsilon_{\hat\xi})\cap M\times(\tau,T).\]
Now in order to transform $\frac{\partial\varphi(\xi,t)}{\partial\eta}$ back to $\mathbb{R}^n$, pick $L>0$, such that\\ $E=\underbrace{[-L,L]\times[-L,L]\times...\times[-L,L]}_
{\mbox{ $(n-1)$ times}}\times[0,L]\subset V$, and
 define $\tilde\varphi$ such that
\[\tilde\varphi(x,t)=-\int_0^{x_n} \nabla_{x}\varphi(\psi(x',z),t)^{T}D(\psi(x',z))\eta(\psi(x',z)) \ dz\quad\forall\ x=(x',z)\in E\]
where $x'\in \underbrace{[-L,L]\times[-L,L]\times...\times[-L,L]}_
{\mbox{ $(n-1)$ times}}$. We know $\varphi\in W^{2,1}_p(\Omega\times(\tau,T))$. Therefore from Lemma $\ref{lp6}$, there exists $0<\alpha<L$ and $K_{\hat\xi}>0$, depending on $\Omega,n,p$ such that \begin{align}\label{6.1} \int_{S_\alpha}\left|\frac{\partial\tilde\varphi((x',\alpha),t)}{\partial x_n}\right|^{r}\ d\sigma dt<K_{\hat\xi}\Vert\varphi\Vert^{(2)}_{p,\Omega\times(\tau,T)} \quad \forall\ 1<r\leq \frac{(n+2)p}{n+2-p}\end{align} where $S_\alpha= E|_{x_n=\alpha}\times(\tau,T)$ and $S_{x_n}= E|_{ 0\leq x_n\leq \alpha}\times(\tau,T)$.
Using the  fundamental theorem of calculus, \begin{align*}
 \int_{E\times(\tau,T)}\left|\frac{\partial\tilde\varphi((x',0),t)}{\partial x_n}\right|^{q}\ d\sigma\  dt&\leq  \int_{S_\alpha}\left|\frac{\partial\tilde\varphi((x',\alpha),t)}{\partial x_n}\right|^{q}\ d\sigma\  dt\\ \nonumber &\quad +q \int_{S_{x_n}}\left|\frac{\partial\tilde\varphi((x',s),t)}{\partial x_n}\right|^{q-1}.\left|\frac{\partial^2\tilde\varphi((x',s),t)}{\partial x_n^2}\right|\ d\sigma\  dt \end{align*}
Using $(\ref{6.1})$,
\begin{align*}
 \int_{E\times(\tau,T)}\left|\frac{\partial\tilde\varphi((x',0),t)}{\partial x_n}\right|^{q} \ d\sigma\  dt& \leq K_{\hat\xi}(\Vert\varphi\Vert^{(2)}_{p,\Omega\times(\tau,T)})^q\\ \nonumber \quad&+q \int_{S_{x_n}}\left|\frac{\partial\tilde\varphi((x',s),t)}{\partial x_n}\right|^{q-1}.\left|\frac{\partial^2\tilde\varphi((x',s),t)}{\partial x_n^2}\right| \ d\sigma \ dt\end{align*}
 Applying H\"older inequality,
 \begin{align*}
 \int_{E\times(\tau,T)}\left|\frac{\partial\tilde\varphi((x',0),t)}{\partial x_n}\right|^{q}\ d\sigma\  dt& \leq K_{\hat\xi}(\Vert\varphi\Vert^{(2)}_{p,\Omega\times(\tau,T)})^q\\ &+ q \left(\int_{S_{x_n}}\left|\frac{\partial\tilde\varphi((x',s),t)}{\partial x_n}\right|^{\frac{(q-1)p}{p-1}}\ d\sigma\  dt\right)^{\frac{p-1}{p}}\left(\int_{S_{x_n}}\left|\frac{\partial^2\tilde\varphi((x',s),t)}{\partial x_n^2}\right|^p \ d\sigma\  dt\right)^{\frac{1}{p}} \end{align*}
 Recall $\frac{\partial^2\tilde\varphi}{\partial x_n^2}\in L_p(S_{x_n})$. So using Lemma $\ref{lp6}$ we have
 \begin{align}\label{grad2}
 \int_{E\times(\tau,T)}\left|\frac{\partial\tilde\varphi((x',0),t)}{\partial x_n}\right|^{q}\ d\sigma\  dt& \leq  \hat K(\Vert  \varphi\Vert_{p,\Omega\times(\tau,T)}^{(2)})^q \end{align}

Now, $M$ is a compact manifold. Therefore there exists set $A=\lbrace P_1,...,P_N\rbrace\subset M$ such that $M\subset\cup_{1\leq i\leq N} B(P_i,\epsilon_{P_i})$. Let $V_i$, $\hat K_i$ and $\alpha_i$ be the open sets and constants respectively obtained above when $\hat\xi=P_i$.  Then,
\begin{align*}
 \left(\int_\tau^T\int_M\left|\frac{\partial\varphi}{\partial \eta}\right|^{q}\ d\sigma\  dt \right)^{\frac{1}{q}}&\leq \left(\sum_{P_i\in A}\int_\tau^T\int_{B(P_i,\epsilon)}\left|\frac{\partial\varphi}{\partial \eta}\right|^{q}\ d\sigma\  dt\right)^{\frac{1}{q}} \\& \leq C\left(\sum_{P_i\in A}\int_\tau^T\int_{V_i|_{x_n=0}}\left| \frac{\partial\tilde\varphi}{\partial x_n}\right|^q\ d\sigma\  dt\right)^{\frac{1}{q}}\\
&\leq C\sum_{P_i\in A}\tilde K_i\Vert  \varphi\Vert_{p, \Omega\times(\tau,T)}^ {(2)}\end{align*}
Therefore, for some $\hat C>0$, depending only upon $p,\tau,T,M$ and $n$, we get
\begin{align*}
\left\Vert\frac{\partial\varphi}{\partial\eta}\right\Vert_{q,M\times(\tau,T)}&\leq \hat C {\left \Vert \varphi\right\Vert}^{(2)}_{p,\Omega\times(\tau,T)} \quad\text{for all}\  1<q\leq \frac{(n+1)p}{n+2-p}
\end{align*}
\end{proof}\\
The following Lemma plays a key role in bootstrapping $L_p$ estimates of solutions to $(\ref{sy5})$.\\
\begin{lemma}\label{adventure}
Assume the hypothesis of Corollary $\ref{needco}$, and suppose $(u,v)$ is the unique, maximal nonnegative solution to $(\ref{sy5})$ and $T_{\max}<\infty$. If $1\leq j\leq k$ and $1\leq i\leq m$, such that $(V_{i,j}1)$ holds, then there exists $K_{T_{\max}}>0$ such that \[ \Vert u_j(\cdot,t)\Vert_{1,\Omega}+\Vert v_i(\cdot,t)\Vert_{1,M}+\Vert u_j\Vert_{1,M\times(0,T_{\max})}\leq K_{T_{\max}} \quad\text {for all } \ 0\leq t<T_{max}.\]
\end{lemma}
\begin{proof}For simplicity, take $\sigma=1$ in $(V_{i,j}1)$. Let $0<T<T_{\max}$, and consider the system \begin{align}\label{aj5} \varphi_t&=- d\Delta \varphi
&( x,t)\in \Omega\times (0,T)\nonumber
\\ d\frac{\partial \varphi}{\partial \eta}&= \alpha \varphi +1 &(x,t)\in M\times (0,T)\\
\varphi&= \varphi_T &x\in\Omega ,\quad t=T\nonumber
\end{align}
where $\alpha$ is given in $(V_{i,j}1)$, $d>0$, and $\varphi_T\in C^{2+\varUpsilon}(\overline\Omega)$ for some $\varUpsilon>0$, is nonnegative and satisfies the compatibility condition \[ d\frac{\partial\varphi_T}{\partial\eta}=\alpha\varphi_T+1 \quad \text{on} \ M\times \lbrace T\rbrace\] From Lemma $\ref{more}$, $\varphi\in C^{2+\varUpsilon,1+\frac{\varUpsilon}{2}}(\overline\Omega\times[0,T])$ and therefore by standard sequential argument $\varphi\in C^{2+\varUpsilon,1+\frac{\varUpsilon}{2}}(M\times[0,T])$. Also, note that $g(s)=\alpha s+1$ satisfies $g(0)\geq 0$. Therefore, Proposition $\ref{p}$ implies $\varphi\geq 0$. Now having enough regularity for $\varphi$ on $M\times[0,T]$, consider \[\Delta_M \varphi=-\frac{1}{\sqrt {det\  g}}\partial_j (g^{ij}\sqrt{det\ g}\ \partial_i\varphi)\] where $g$ is the metric on $M$ and $g^{i,j}$ is $i$th row and $j$th column entry of the inverse of matrix associated to metric $g$. Further let $\tilde\vartheta=-\varphi_t-\tilde d\Delta_M\varphi$. 
Then,
\begin{align*}
\int_0^T\int_{M} v_i\tilde\vartheta & =\int_0^T\int_{\Omega} u_j(-\varphi_t-d\Delta\varphi) +\int_0^T\int_{M} v_i(-\varphi_t-\tilde d\Delta_M \varphi) \\
& =\int_0^T\int_{\Omega} \varphi ({u_j}_t-d\Delta u_j) +\int_0^T\int_{M} \varphi({v_i}_t-\tilde d\Delta_M v_i) -d\int_0^T\int_{M} u_j\frac{\partial\varphi}{\partial\eta}+d\int_0^T\int_{M} \frac{\partial u_j}{\partial\eta}\varphi\\ &\quad+\int_{\Omega} u_j(x,0)\varphi(x,0)+\int_{M} v_i(\zeta,0)\varphi(x,0)-\int_{\Omega} u_j(x,T)\varphi_T-\int_{M} v_i(\zeta,T)\varphi_T
\end{align*}
Using $d\frac{\partial\varphi}{\partial\eta}=\alpha\varphi+1$ 
\begin{align}
\int_0^T\int_{M} u_j&\leq \int_0^T\int_{\Omega} \varphi H_j(u)+\int_0^T\int_M  (F_i(u,v)+G_j(u,v))\varphi \nonumber\\& \quad\quad+\int_{\Omega} u_j(x,0)\varphi(x,0)+\int_{M} v_i(\zeta,0)\varphi(x,0)-\int_0^T\int_{M} v_i\tilde\vartheta\nonumber
\end{align}
Using $(V_{i,j}1)$, 
\begin{align}\label{onbd}
\int_0^T\int_{M} u_j&\leq \int_0^T\int_{\Omega}\beta \varphi(u_j+1)+\int_0^T\int_M \alpha (v_i+1)\varphi\\&\quad\quad+\int_{\Omega} u_j(x,0)\varphi(x,0)+\int_{M} v_i(\zeta,0)\varphi(x,0))-\int_0^T\int_{M} v_i\tilde\vartheta\nonumber
\end{align}
Now, integrating the $u_j$ equation over $\Omega$ and the $v_i$ equation over $M$, 
\begin{align}\label{upper}
\frac{d}{dt}\left(\int_{\Omega} u_j+ \int_{M} v_i\right)&=d\int_{\Omega} \Delta u_j +\int_{\Omega}H_j(u) +\tilde d\int_{M} \Delta v_i + \int_{M} F_j(u,v) \nonumber\\ &\leq  \nonumber\beta\int_{\Omega}(u_j+1)+\int_{M}( G_j(u,v)+F_i(u,v) )\\ &\leq  \beta\int_{\Omega}(u_j+1) + \alpha \int_{M}(u_j+v_i+1)
\end{align}
 Integrating $(\ref{upper})$ over $(0, t)$ with $0<t\leq T<T_{\max}$, and using $(\ref{onbd})$, gives \\
\begin{align}\label{upper2}
\int_{\Omega} u_j(x,t)+ \int_{M} v_i (\zeta,t)&\leq \tilde\beta\int_{0}^{t}\int_{\Omega}u_j+\tilde\alpha\int_0^t\int_M v_i+\tilde L(t)
\end{align}
where \begin{align*}
\tilde L(t)=\alpha|M|t&+\beta|\Omega|t+\alpha\beta\Vert\varphi\Vert_{1,\Omega\times(0,t)}+{\alpha}^2\Vert\varphi\Vert_{1,M\times(0,t)}+{\alpha}\Vert u_j(x,0)\Vert_{1,\Omega}\cdot\Vert\varphi(x,0)\Vert_{\infty,\Omega}\\&+\Vert v_i(\zeta,0)\Vert_{1,M}+{\alpha}\Vert v_i(\zeta,0)\Vert_{1,M}\cdot \Vert\varphi(x,0)\Vert_{\infty, M}+\Vert u_j(x,0)\Vert_{1,\Omega}
\end{align*}
\[\tilde\alpha (t) ={\alpha}^2\Vert\varphi\Vert_{\infty, M\times(0,t)}+\alpha+\alpha\Vert\tilde\theta\Vert_{\infty,M\times(0,t)} \quad \text{and} \quad \tilde \beta (t) =\beta+\alpha\beta\Vert\varphi\Vert_{\infty,\Omega\times(0,t)}\]
Applying Generalized Gronwall's inequality to $(\ref{upper2})$ gives the bound for the first two integrals on the RHS of $(\ref{upper2})$, and then substituting this bound gives
\begin{align*}
 \int_{\Omega} u_j(x,t) + \int_{M} v_i(\zeta,t)&\leq \tilde L(t)+\int_{0}^{t} (\tilde \alpha(s)+\tilde\beta(s)) \tilde L(s) \exp\left(\int_{s}^{t} \tilde\alpha(r)+\tilde\beta(r) dr\right)\ ds \\&\leq C_{T_{\max}}
\end{align*}
for all $0\leq t< T<T_{max}$. Substituting this estimate of $u_j$ on $\Omega$ and $v_i$ on $M$ in $(\ref{onbd})$ yields
\begin{align*}
\int_0^T\int_{M} u_j &\leq \beta\left(\Vert \varphi\Vert_{\infty, \Omega\times(0,T)}\Vert u_j\Vert_{1,\Omega\times(0,T)}+|\Omega|T\Vert \varphi\Vert_{\infty, \Omega\times(0,T)}\right)\\
&\quad\quad+\alpha\left(\Vert \varphi\Vert_{\infty, M\times(0,T)}\Vert v_i\Vert_{1,M\times(0,T)}+|M|T\Vert \varphi\Vert_{\infty, M\times(0,T)}\right)\\&\quad\quad+\Vert u_j(\cdot,0)\Vert_{1,\Omega}\Vert\varphi(\cdot,0)\Vert_{\infty,\Omega}+\Vert v_i(\cdot,0)\Vert_{1,M}\Vert\varphi(\cdot,0)\Vert_{\infty,M}+\Vert v_i\Vert_{1,M}\Vert\tilde\theta\Vert_{\infty,M}
\end{align*}
Since $T<T_{\max}$ is arbitrary, the conclusion of the theorem holds.
\end{proof}\\
\begin{lemma}\label{implication}
Assume the hypothesis of Corollary $\ref{needco}$ holds. Suppose $(u,v)$ is the unique, maximal nonnegative solution to $(\ref{sy5})$ and $T_{\max}<\infty$. If $1\leq j\leq k$ and $1\leq i\leq m$, such that $(V_{i,j}1)$ and $(V_{i,j}2)$ holds, and for $q> 1$, $v_i\in L_q(M\times(0,T_{\max}))$, then  $u_j\in L_{q}(M\times(0,T_{\max}))$ and $u_j\in L_{q}(\Omega\times(0,T_{\max}))$.
\end{lemma}
\begin{proof}
Let $0<t<T\leq T_{\max}$. Multiplying the ${u_j}_t$ equation by $u_j^{q-1}$, we get 
\begin{align}
\int_{0}^{t}\int_{\Omega}u_j^{q-1}{u_j}_t &=d\int_{0}^{t}\int_{\Omega}u_j^{q-1}\Delta u_j+\int_{0}^{t}\int_{\Omega}u_j^{q-1}H_j(u)\nonumber\\ &= d\int_{0}^{t}\int_{M}u_j^{q-1}{\frac{\partial u_j}{\partial\eta}}-d\int_{0}^{t}\int_{\Omega}(q-1)u_j^{q-2}{|\nabla u_j|^{2}}+ \int_{0}^{t}\int_{\Omega}u_j^{q-1}H_j(u)\nonumber
\end{align}
Using $(V_{i,j}2)$
\begin{align}\label{MM3}
 \int_{\Omega}\frac{u_j^{q}}{q}+d\int_{0}^{t}\int_{\Omega}\frac{4(q-1)}{q^{2}}{|\nabla u_j^{\frac{q}{2}}|^{2}}&\leq K_g\int_{0}^{t}\int_{M}u_j^{q-1}(u_j+v_i+1)+\beta\int_{0}^{t}\int_{\Omega} (u_j+1)u_j^{q-1}\nonumber\\ &\quad \quad + \int_{\Omega}\frac{{u_j}_0^{q}}{q}\nonumber\\
&\leq K_g\left(\int_{0}^{t}\int_{M}u_j^{q}+v_i u_j^{q-1}+u_j^{q-1}\right)+\beta\left(\int_{0}^{t}\int_{\Omega} u_j^q+u_j^{q-1}\right)\nonumber\\ &\quad \quad + \int_{\Omega}\frac{{u_j}_0^{q}}{q}
\end{align}
Applying Young's inequality in $(\ref{MM3})$
\begin{align}\label{MM84}
 \int_{\Omega}\frac{u_j^{q}}{q}+d\int_{0}^{t}\int_{\Omega}\frac{4(q-1)}{q^{2}}{|\nabla u_j^{\frac{q}{2}}|^{2}}&\leq K_g\left(\frac{3q-2}{q}\right)\int_{0}^{t}\int_{M}u_j^{q}+\left(\beta+t|\Omega|\frac{\beta}{q}\right)\int_{0}^{t}\int_{\Omega} u_j^q\nonumber\\ &\quad \quad + \int_{\Omega}\frac{{u_j}_0^{q}}{q}+ K_g\left(\frac{1}{q}\right)\int_{0}^{t}\int_{M}v_i^{q}+\frac{t|M|}{q}
\end{align}
Also, for $1< q\leq \infty$, for all $\epsilon>0$ and $t\leq T\leq T_{max}$, from Lemma $\ref{i}$, for $v=u^{\frac{q}{2}}$ there exists $C_\epsilon>0$ such that,
\begin{eqnarray}\label{MM}
\int_{0}^{t}\int_{M}u_j^{q}\leq C_{\epsilon}\int_{0}^{t}\int_{\Omega}u_j^{q}+\epsilon\int_{0}^{t}\int_{\Omega}{|\nabla u_j^{\frac{q}{2}}|^{2}}\quad
\end{eqnarray}
Using $(\ref{MM})$ and $(\ref{MM84})$ for appropriate $\epsilon>0$, gives
\begin{align}
\frac{1}{q}\frac{d}{dt}\int_{0}^{t}\int_{\Omega} u_j^{q}\leq \tilde K_1 \int_{0}^{t}\int_{\Omega} u_j^{q}+\tilde K_2(T)
\end{align} 
for \[\tilde K_2(T)=  K_g\left(\frac{1}{q}\right)\int_{0}^{T}\int_{M}v_i^{q}+\frac{T|M|}{q}\] and  $\tilde K_1>0$ depending on t, where $t\leq T\leq T_{max}$. Therefore from Gronwall's Inequailty 
\begin{align}\label{strong}
\int_{\Omega} {u_j}^{q}(x,t)&\leq {\tilde K_2(T)}+\int_{0}^{T} {\tilde K_1(s)} {\tilde K_2(s)} \exp \left( \int_{s}^{T} \tilde K_1(r) dr\right) ds
\end{align}
To obtain estimates on boundary, we use $(\ref{MM84})$ to obtain
\begin{eqnarray}\label{MM4}
\epsilon \int_{0}^{T}\int_{\Omega}{|\nabla u_j^{\frac{q}{2}}|^{2}}&\leq \left(\frac{q^2}{4d(q-1)}\right)3K_g \epsilon \int_{0}^{T}\int_{M}u_j^{q}+\epsilon\left(\frac{q^2}{4d(q-1)}\right)\left(\beta+T|\Omega|\frac{\beta}{q}\right)\int_{0}^{T}\int_{\Omega} u_j^q\nonumber\\ &\quad \quad +\left(\frac{q^2}{4d(q-1)}\right)\left(\epsilon \int_{\Omega}\frac{{u_j}_0^{q}}{q}+\epsilon K_g\left(\frac{1}{q}\right)\int_{0}^{T}\int_{M}v_i^{q}+\epsilon\frac{T|M|}{q}\right)
\end{eqnarray}
Using $(\ref{MM})$, $(\ref{MM4})$ and $(\ref{strong})$ we have,
\begin{align*}
\int_{0}^{T}\int_{M}u_j^{q}&\leq C_{\epsilon}\int_{0}^{T}\int_{\Omega}u_j^{q}+3K_g\left(\frac{q^2}{4d(q-1)}\right)\epsilon \int_{0}^{T}\int_{M}u_j^{q}+\epsilon\left(\frac{q^2}{4d(q-1)}\right)\left(\beta+T|\Omega|\frac{\beta}{q}\right)\int_{0}^{T}\int_{\Omega} u_j^q\nonumber\\ &\quad \quad +\left(\frac{q^2}{4d(q-1)}\right)\left(\epsilon \int_{\Omega}\frac{{u_j}_0^{q}}{q}+\epsilon K_g\left(\frac{1}{q}\right)\int_{0}^{T}\int_{M}v_i^{q}+\epsilon\frac{T|M|}{q}\right)
\end{align*}
Now choosing $\epsilon$ such that \[1-3K_g\left(\frac{q^2}{4d(q-1)}\right)\epsilon>0\] and using the estimate above for $u_j$ on $(\Omega\times(0,T))$, we have $u_j\in L_q(M\times(0,T))$. Since T is arbitrary, $u_j\in L_q(M\times(0,T_{max}))$
\end{proof}\\

\begin{lemma}\label{global_time1}
Assume the hypothesis of Corollary $\ref{needco}$, and suppose $(u,v)$ is the unique, maximal nonnegative solution to $(\ref{sy5})$ and $T_{\max}<\infty$. If $1\leq j\leq k$ and $1\leq i\leq m$ so that $(V_{i,j}1)$ and $(V_{i,j}2)$ hold, then for all $p>1$ and $0<T<T_{\max}$, there exists $C_{p,T}>0$, such that  
\begin{align*}
\Vert u_j\Vert_{p,\Omega\times(0,T_{max})}&+\Vert v_i\Vert_{p,M\times(0,T_{max})}\\&\leq C_{p,T_{max}}\left(\Vert u_j\Vert_{1,M\times(0,T_{max})}+\Vert u_j\Vert_{1,\Omega\times(0,T_{max})}+\Vert v_i\Vert_{1,M\times(0,T_{max})}\right)
\end{align*}
\end{lemma}
\begin{proof}
First we show there exists  $r>1$ such that if $q\geq 1$ such that $u_j\in L_{q}(\Omega\times(0,T_{max}))$ and $v_i\in L_{q}(M\times(0,T_{max}))$ then $u_j\in L_{rq}(\Omega\times(0,T_{max}))$ and $v_i\in L_{rq}(M\times(0,T_{max}))$. 
Consider the system ($\ref{aj2}$) and ($\ref{ajj3}$) with $\kappa_1=0$, $\kappa_2=1$, $\tilde\vartheta\geq 0$, $\tilde\vartheta\in L_{p}{(M\times(0,T_{max}))}$ with $ {\Vert \tilde\vartheta\parallel}_{p,(M\times(0,T_{max}))}=1$,   $\vartheta\geq 0$, and  $\vartheta\in L_{p}{(\Omega\times(0,T_{max}))}$ with $ {\Vert \vartheta\parallel}_{p,(\Omega\times(0,T_{max}))}=1$.
 Multiplying $u_j$ with $\vartheta$ and $v_i$ with $\tilde\vartheta$ and for $0<T\leq T_{max}$, integrating over $\Omega\times(0,T)$ and $M\times(0,T)$ respectively, gives
\begin{align*}
\int_0^T\int_{\Omega} u_j\vartheta +\int_0^T\int_{M} v_i\tilde\vartheta &=\int_0^T\int_{\Omega} u_j(-\varphi_t-d\Delta\varphi) +\int_0^T\int_{M} v_i(-\Psi_t-\tilde d\Delta_M \Psi) \\
&=\int_0^T\int_{\Omega} \varphi ({u_j}_t-d\Delta u_j) +\int_0^T\int_{M} \Psi({v_i}_t-\tilde d\Delta_M v_i) \\ &-d\int_0^T\int_{M} u_j\frac{\partial\varphi}{\partial\eta}+d\int_0^T\int_{M} \frac{\partial u_j}{\partial\eta}\varphi+\int_{\Omega} u_j(x,0)\varphi(x,0)\\&+\int_{M} v_i(x,0)\Psi(x,0)-\int_{M}v_i(x,T)\Psi(x,T)-\int_{\Omega}u_j(x,T)\varphi(x,T)
\end{align*}
Since $\Psi(x,T)=0$ and $\varphi(x,T)=0$,
\begin{align*}
\int_0^T\int_{\Omega} u_j\vartheta +\int_0^T\int_{M} v_i\tilde\vartheta  &\leq \int_0^T\int_{\Omega} \varphi H_j(u)+\int_0^T\int_M  (F_j(u,v)+G_i(u,v))\Psi\\& \quad\quad-d\int_0^T\int_{M} u_j\frac{\partial\varphi}{\partial\eta}+\int_{\Omega} u_j(x,0)\varphi(x,0)\\&\quad\quad+\int_{M} v_i(x,0)\Psi(x,0)
\end{align*}
Using $(V_{i,j}1)$,
\begin{align}\label{ineq}
\int_0^T\int_{\Omega} u_j\vartheta +\int_0^T\int_{M} v_i\tilde\vartheta \nonumber &\leq \int_0^T\int_{\Omega}\beta \varphi(u_j+1)+\int_0^T\int_M  \alpha(u_j+v_i+1)\Psi\\ \nonumber& \quad\quad-d\int_0^T\int_{M} u_j\frac{\partial\varphi}{\partial\eta}+\int_{\Omega} u_j(x,0)\varphi(x,0)\\&\quad\quad+\int_{M} v_i(x,0)\Psi(x,0)
\end{align}
Now we break the argument in two cases.

Case 1: Suppose $q=1$. Then $u_j\in L_1(\Omega\times(0,T_{max}))$ and $u_j,v_i\in L_1(M\times(0,T_{max}))$. Let $\epsilon>0$ and set $ p=n+2+\epsilon$. Set $p'=\frac{n+2+\epsilon}{n+1+\epsilon}$ (conjugate of $p$). Remarks $\ref{hol}$ and $\ref{holl}$, and Lemma $\ref{adventure}$ imply all of the integrals on the right hand side of $(\ref{ineq})$ are finite. Application of H\"older's inequality in $(\ref{ineq})$, yields $v_i\in L_{p'}(M\times(0,T))$, and there exists $C_{p,T}>0$ such that
\begin{align*}
\Vert u_j\Vert_{p',\Omega\times(0,T)}+\Vert v_i\Vert_{p',M\times(0,T)}&\leq C_{p,T}(\Vert u_j\Vert_{1,\Omega\times(0,T_{max})}+\Vert v_i\Vert_{1,M\times(0,T_{max})}+\Vert u_j\Vert_{1,M\times(0,T_{max})})
\end{align*}
Since $T\leq T_{max}$ is arbitrary, therefore, Lemma $\ref{implication}$ implies $u_j\in L_{p'}(M\times(0,T_{max}))$. So for this case, $r=\frac{n+2+\epsilon}{n+1+\epsilon}$.

Case 2: Suppose $q>1$ such that $u_j\in L_q(\Omega\times(0,T_{max}))$ and $u_j,v_i\in L_q(M\times(0,T_{max}))$. \\
Recall $p>1$, $0\leq \tilde\vartheta\in L_{p}{(M\times(0,T_{max}))}$ with $ {\Vert \tilde\vartheta\parallel}_{p,(M\times(0,T_{max}))}=1$ and  $0\leq\vartheta\in L_{p}{(\Omega\times(0,T_{max}))}$ with $ {\Vert \vartheta\parallel}_{p,(\Omega\times(0,T_{max}))}=1$. Also $p'=\frac{p}{p-1}$, $q'=\frac{q}{q-1}$.
 Note $T\leq T_{max}$ is arbitrary. Applying H\"older's inequality  in $(\ref{ineq})$ and using Lemma $\ref{flat}$, yields  
\begin{align*}
\Vert u_j\Vert_{p',\Omega\times(0,T_{max})}&+\Vert v_i\Vert_{p',M\times(0,T_{max})}\\&\leq C_{p,T_{max}}(\Vert u_j\Vert_{q,\Omega\times(0,T_{max})}+\Vert v_i\Vert_{q,M\times(0,T_{max})}+\Vert u_j\Vert_{q,M\times(0,T_{max})})
\end{align*}
provided $p'\leq \frac{(n+2)q}{n+1}$. So, in this case, $r=\frac{(n+2)}{n+1}$.

Now, by repeating the above argument for $rq$ instead of $q$, we get  $v_i\in L_{r^{m}q}(M\times(0,T_{max}))$, $u_j\in L_{r^{m}q}(\Omega\times(0,T_{max}))$,  for all $m>1$. As $r>1$, $\lim \limits_{m\rightarrow\infty}{r^{m}q}\rightarrow\infty$, and as a result, $v_i\in L_{p}(M\times(0,T_{max}))$ for all $p>1$. Hence from Lemma $\ref{implication}$, $u_j\in L_{p}(M\times(0,T_{max}))$ and $u_j\in L_{p}(\Omega\times(0,T_{max}))$ for all $p>1$, and there exists $C_{p,T}>0$ such that  
\begin{align*}
\Vert u_j\Vert_{p,\Omega\times(0,T)}+\Vert v_i\Vert_{p,M\times(0,T)}&\leq C_{p,T}\left(\Vert u_j\Vert_{q,M\times(0,T_{max})}+\Vert u_j\Vert_{q,\Omega\times(0,T_{max})}+\Vert v_i\Vert_{q,M\times(0,T_{max})}\right)
\end{align*} Again as $T\leq T_{max}$ is arbitrary, we get
\begin{align*}
\Vert u_j\Vert_{p,\Omega\times(0,T_{max})}&+\Vert v_i\Vert_{p,M\times(0,T_{max})}\\&\leq C_{p,T_{max}}\left(\Vert u_j\Vert_{1,M\times(0,T_{max})}+\Vert u_j\Vert_{1,\Omega\times(0,T_{max})}+\Vert v_i\Vert_{1,M\times(0,T_{max})}\right)
\end{align*}
\end{proof}

\subsection{Global Existence}
\quad\\

{\bf Proof of Theorem $\ref{great}$:} From Theorem $\ref{lo}$ and Corollary $\ref{needco}$, we already have a componentwise nonnegative, unique, maximal solution of $(\ref{sy5})$. If $T_{\max}=\infty$, then we are done. So, by way of contradiction assume $T_{\max}<\infty$. From Lemma $\ref{global_time1}$, we have $L_p$ estimates for our solution for all $p\geq 1$, on $\Omega\times(0,T_{\max})$ and $M\times(0,T_{\max})$. We know from $(V_{i,j}2)$ and $(V_{i,j}3)$ that $F_j$ and $G_i$ are polynomially bounded above for each $i$ and $j$. Let $U$ and $V$ solve
\begin{align}\label{comp1} U_t&=d_j \Delta U+\beta(u_j+1)&
( x,t)\in \Omega\times(0,T_{max})
\nonumber\\V_t&=\tilde d_i\Delta_M V+K_f(u_j+v_i+1)^l& (x,t)\in M\times(0,T_{max})\nonumber\\ d_j\frac{\partial U}{\partial \eta}&=K_g(u_j+v_i+1)& (x,t)\in M\times(0,T_{max})\\
U&=U_0& x\in\Omega ,\quad t=0\nonumber\\V&=V_0&x\in M ,\quad t=0\nonumber\end{align}
Here, $d_j$ and $\tilde d_i$ are the $j$th and $i$th column entry of diagonal matrix $D$ and $\tilde D$ respectively. Also, $U_0$ and $V_0$ satisfy the compatibility condition, are component-wise nonnegative functions, and $(u_0)_j\leq U_0$ and $(v_0)_i\leq V_0$. For all $q\geq 1$, $K_f(u_j+v_i+1)^l$ and $K_g(u_j+v_i+1)$ lie in $L_q(M\times(0,T_{max}))$. Using Theorem $\ref{n}$, the solution of $(\ref{comp1})$ is sup norm bounded. Therefore, by the Maximum Principle \cite{RefWorks:57}, the solution of  $(\ref{sy5})$ is bounded for finite time. Therefore Theorem $\ref{lo}$ implies $T_{max}=\infty$. $\square$

				

\section{Examples and an Open Question}
In this section we give some examples to support our theory.\\
\begin{example}
As described in \cite{RefWorks:142}, during bacterial cytokinesis, a proteinaceous contractile, called the $Z$ ring assembles in the cell middle. The $Z$ ring moves to the membrane and contracts, when triggered, to form two identical daughter cells. Positiong the $Z$ ring in the middle of the cell involves two independent processes, referred to as Min system inhibition and nucleoid occlusion (\cite{RefWorks:140}, \cite{RefWorks:141} Sun and Margolin 2001). In this example, we only discuss the Min system inhibits process. The Min system involves proteins MinC, MinD and MinE (\cite{RefWorks:144} Raskin and de Boer 1999). MinC inhibits $Z$ ring assembly while the action of MinD and MinE serve to exclude MinC from the middle of cell region. This promotes the assembly of the $Z$ ring at the middle of the cell. In \cite{RefWorks:142} the authors considered the Min subsystem involving 6 chemical reactions and 5 components, under specific rates and parameters and performed a numerical investigation using a finite volume method on a one dimensional mathematical model.  Table 7.1 shows the assumed chemical reactions. The model was developed in \cite{RefWorks:142} within the context of a cylindrical cell consisting of 2 subsystems; one involving Min oscillations and the other involving FtsZ reactions. The Min subsystem consists of ATP-bound cytosolic MinD, ADP-bound cytosolic MinD, membrane-bound MinD, cytosolic MinE, and membrane bound MinD:MinE complex. Those are denoted $D_{cyt}^{ATP}$, $D_{cyt}^{ADP}$, $D_{mem}^{ATP}$, $E_{cyt}$, and $E:D_{mem}^{ATP}$, respectively. This essentially constitutes the one dimensional version of the problem. These Min proteins  react with certain reaction rates that are illustrated in Table 7.1. 
\begin{table}[ht]\label{table:nonlin}
\caption{Reactions and Reaction Rates} 
\centering 
\begin{tabular}{|ccc| }
\hline\hline                       
Chemicals & Reactions & Reaction Rates\\ [0.5ex] 
\hline 
& & \\            
Min D  & $D^{ADP}_{cyt}\xrightarrow{k_{1}} D_{cyt}^{ATP}$ & $ R_{exc}=k_{1}[D_{cyt}^{ADP}]$ \\[1ex]
Min D & $D_{cyt}^{ATP}\xrightarrow{k_{2}} D_{mem}^{ATP}$ & $R_{Dcyt}=k_{2}[D_{cyt}^{ATP}]$\\ [1ex]
&$D_{cyt}^{ATP}\xrightarrow{k_{3}[D_{mem}^{ATP}]} D_{mem}^{ATP}$ & $R_{Dmem}=k_{3}[D_{mem}^{ATP}][D_{cyt}^{ATP}]$\\ [1ex]  
Min E  &$E_{cyt}+D_{mem}^{ATP}\xrightarrow{k_{4}} E:D_{mem}^{ATP}$ &$ R_{Ecyt}=k_{4}[E_{cyt}] [D_{mem}^{ATP}]$\\ [1ex] 
& $E_{cyt}+D_{mem}^{ATP}\xrightarrow{k_{5}[E:D_{mem}^{ATP}]^2} E:D_{mem}^{ATP}$ & $R_{Emem}=k_{5}[D_{mem}^{ATP}][E_{cyt}][E:D_{mem}^{ATP}]^2$\\[1ex]
Min E & $ E:D_{mem}^{ATP} \xrightarrow{k_{6}} E+D_{cyt}^{ADP}$& $R_{exp}=k_{6}[E:D_{mem}^{ATP}]$\\[1ex]
\hline 
\end{tabular}
\end{table} These reactions lead to five component model with $(u,v)=(u_1,u_2,u_3,v_1,v_2)$, where
\[u=\begin{pmatrix}u_1\\u_2\\u_3\end{pmatrix}=\begin{pmatrix} \left[D_{cyt}^{ATP}\right]\\ \left[D_{cyt}^{ADP}\right]\\ \left[E_{cyt}\right] \end{pmatrix}, {v}=\begin{pmatrix}v_1\\v_2\end{pmatrix}=\begin{pmatrix}\left[D_{mem}^{ATP}\right]\\ \left[E:D_{mem}^{ATP}\right]\end{pmatrix} \]
\[\tilde D=\begin{pmatrix} \sigma_{Dmem} & 0 \\ 0 &  \sigma_{E:Dmem} \end{pmatrix},\quad  D=\begin{pmatrix} \sigma_{Dcyt} & 0 & 0\\ 0 &  \sigma_{ADyct} & 0\\0 & 0 & \sigma_{Ecyt} \end{pmatrix} \]\\
 \[{G(u,v)}=\begin{pmatrix}G_1(u,v)\\G_2(u,v)\\G_3(u,v)\end{pmatrix}=\begin{pmatrix}- R_{Dcyt}-R_{Dmem}\\ R_{exp}\\ R_{exp}-R_{Ecyt}-R_{Emem}\end{pmatrix}=\begin{pmatrix}- k_2 u_1-k_3v_1u_1\\ k_6v_2\\ k_6v_2-k_4u_3v_1-k_5v_1u_3{v_2}^2\end{pmatrix}, \]
 \[ {F(u,v)}=\begin{pmatrix}F_1(u,v)\\F_2(u,v)\end{pmatrix}=\begin{pmatrix} R_{Dcyt}+R_{Dmem}-R_{Ecyt}-R_{Emem}\\ -R_{exp}+R_{Ecyt}+R_{Emem}\end{pmatrix}= \begin{pmatrix}k_2u_1+k_3v_1u_1-k_4u_3v_1-k_5v_1u_3{v_2}^2\\ -k_6v_2+k_4u_3v_1+k_5v_1u_3{v_2}^2\end{pmatrix}, \] \[ {H(u)}=\begin{pmatrix}H_1(u)\\H_2(u)\\H_3(u)\end{pmatrix}=\begin{pmatrix} R_{exc}\\ -R_{exc}\\ 0\end{pmatrix}=\begin{pmatrix} k_1u_2\\ -k_1u_2\\ 0\end{pmatrix}, \] and
$ u_0=( {u_0}_j)\in W_p^{2}(\Omega)$, $v_0= ({v_0}_i)\in W_p^{2}(M)$ are componentwise nonnegative functions with $p>n$.  Also, $u_0 $  and $v_0$ satisfy the compatibility condition\[ D{\frac{ \partial {u_0}}{\partial \eta}} =G(u_0,v_0)\quad \text{on $M.$}\] Here expressions of the form $k_{\alpha}$ and $\sigma_{\beta}$ are positive constants. Note $F, G$ and $H$ are quasi positive functions. In the multidimensional setting, the concentration densities satisfy the reaction-diffusion system given by
\begin{align*}
 u_t\nonumber&=  D \Delta u+H(u)
 & x\in \Omega, \quad&0<t<T
\\\nonumber v_t&=\tilde D \Delta_{M} v+F(u,v)& x\in M,\quad& 0<t<T\\ D\frac{\partial u}{\partial \eta}&=G(u,v) & x\in M, \quad&0<t<T\\\nonumber
u&=u_0  &x\in\Omega ,\quad& t=0\\\nonumber v&=v_0 & x\in M ,\quad &t=0\end{align*}Our local existence result holds for any number of finite components. Therefore, from Theorem $\ref{lo}$, this system has a unique maximal componentwise nonnegative solution. In this example, if we take two specific components at a time, we are able to obtain $L_p$ estimates for each of the components. For that purpose we apply our results to $(u_3,v_2)$, $u_2$ and $(u_1,v_1)$. In order to prove global existence, we assume $T_{max}<\infty$. Otherwise, we are done.

Consider $(u_3,v_2)$. It is easy to see that for $j=3$ and $i=2$, the hypothesis of  Lemma $\ref {global_time1}$ is satisfied, since $G_3+F_2\leq 0$, $G_3$ is linearly bounded, and $H_3=0$. As a result, $u_3\in L_p(\Omega_{T_{max}})$ and $v_2\in L_p(M_{T_{max}})$ for all $p>1$. Using Theorem  $\ref{n}$ and the comparison principle, $u_2$ is H\"{o}lder continuous on $\Omega_{T_{\max}}$ for $p>n+1$. Finally, consider $(u_1,v_1)$. Clearly  for $j=1$ and $i=1$, the hypothesis of  Lemma $\ref {global_time1}$ is satisfied, since $G_1+F_1\leq 0$, $G_1$ is linearly bounded, and $H_1$ is bounded.  Therefore, $u_1\in L_p(\Omega_{T_{max}})$ and $v_1\in L_p(M_{T_{max}})$ for all $p>1$. 

We already have for all $1\leq i\leq 2$ and $1\leq j\leq 3$, $(u_j, v_i)\in L_p(\Omega_{T_{max}})\times L_p(M_{T_{max}})$ for all $p\geq1$. Therefore there exists $\tilde p>1$ such that $G_j\in L_{\tilde p}(\Omega_{T_{max}})$ for all $p\geq\tilde p$, and $F_i\in L_{\tilde p}(M_{T_{max}})$ for all $p\geq\tilde p$. Consequently from Theorem $\ref{n}$, the solution is bounded, which contradicts the conclusion of Theorem $\ref{lo}$. As a result, the system has a global solution.
\end{example}

\begin{example2}
Consider the model considered by R\"{a}tz and R\"{o}ger\cite{RefWorks:99} for signaling networks. They formulated a mathematical model that couples reaction-diffusion in the inner volume to a reaction-diffusion system on the membrane via a flux condition. More specifically, consider the system (3.1) with $k=1, m=2$, where \[{G(u,v)}=-q=-b_6 \frac{|B|}{|M|}u(c_{max}-v_1-v_2)_{+}+b_{-6} v_2,\quad  {H(u)}=0\]
 \[ {F(u,v)}=\begin{pmatrix}F_1(u,v)\\F_2(u,v)\end{pmatrix}=\begin{pmatrix}k_1v_2g_0\left(1-\frac{K_5v_1g_0}{1+K_5v_1}\right)+k_2v_2\frac{K_5v_1g_0}{1+K_5v_1}-k_3\frac{v_1}{v_1+k_4}\\-k_1v_2g_0\left(1-\frac{K_5v_1g_0}{1+K_5v_1}\right)-k_2v_2\frac{K_5v_1g_0}{1+K_5v_1}+k_3\frac{v_1}{v_1+k_4} +q \end{pmatrix} \]  and
$ u_0=( {u_0}_j)\in W_p^{(2)}(\Omega)$, $v_0= ({v_0}_i)\in W_p^{(2)}(M)$ with $p>n$ are componentwise nonnegative. Also, $u_0 $  and $v_0$ satisfy the compatibility condition\[ D{\frac{ \partial {u_0}}{\partial \eta}} =G(u_0,v_0)\quad \text{on $M.$}\] Here $k_\alpha, K_\alpha, g_0, c_{max}, b_{-6}$ are same positive constants as described in \cite{RefWorks:99}. We note $F, G$ and $H$ are quasi positive functions. From Theorem $\ref{lo}$, this system has a unique componentwise nonnegative maximal solution. In order to get global existence, we assume $T_{max}<\infty$. In order to obtain $L_p$ estimates for each of the components, consider $(u,v_2)$. It is easy to see that $G+F_2\leq k_3$, $H=0$, and $G$ is linearly bounded above. So the hypothesis of  Lemma $\ref {global_time1}$  is satisfied. As a result, $u\in L_p(\Omega_{T_{max}})$ and $v_2\in L_p(M_{T_{max}})$ for all $p>1$. Now $v_1$, satisfies the hypothesis of Theorem  $\ref{3}$. Therefore $v_1\in W^{2,1}_p(M_{T_{max}})$ for all $p>1$. 

We already have for all $1\leq i\leq 2$, $(u, v_i)\in L_p(\Omega_{T_{max}})\times L_p(M_{T_{max}})$ for all $p\geq1$. Therefore $G\in L_{ p}(\Omega_{T_{max}})$ for all $p\geq 1$, and $F_i\in L_{ p}(M_{T_{max}})$ for all $p\geq 1$. Consequently, from Theorem $\ref{n}$, the solution is bounded, which contradicts the conclusion of Theorem $\ref{lo}$. As a result the system has a global solution.
\end{example2}

\vspace{.1cm}
\begin{example3}
We look at a simple model to illustrate an interesting open question. Consider the system 
 \begin{align}\label{solvable}
 u_t\nonumber&= \Delta u
 & x\in \Omega, \quad&0<t<T
\\\nonumber v_t&=\Delta_{M} v+u^2 v^2& x\in M,\quad& 0<t<T\\ \frac{\partial u}{\partial \eta}&=-u^2 v^2 & x\in M, \quad&0<t<T\\\nonumber
u&=u_0  &x\in\Omega ,\quad& t=0\\\nonumber v&=v_0 & x\in M ,\quad &t=0\end{align}
where $u_0$ and $v_0$ are nonnegative and smooth, and satisfy the compatibility condition. Clearly $H(u)=0$, $G(u,v)=u^2v^2$ and $F(u,v)=-u^2v^2$ satisfy the hypothesis of Theorem 3.3 with  $F+G\leq 0$ and $G(u,v)\leq 0$. Therefore $(\ref{solvable})$ has a unique global componentwise nonnegative global solution. However, suppose we make a small change, and consider the system
\begin{align}\label{unsolvable}
 u_t\nonumber&= \Delta u
 & x\in \Omega, \quad&0<t<T
\\\nonumber v_t&=\Delta_{M} v-u^2 v^2& x\in M,\quad& 0<t<T\\ \frac{\partial u}{\partial \eta}&=u^2 v^2 & x\in M, \quad&0<t<T\\\nonumber
u&=u_0  &x\in\Omega ,\quad& t=0\\\nonumber v&=v_0 & x\in M ,\quad &t=0\end{align}
Then we can show there exists a unique maximal componentwise nonnegative solution. We can also obtain $L_1$ estimates for $u$ and $v$. Furthermore, it is easy to see that $v$ is uniformly bounded. But our theory cannot be used to determine whether $(\ref{unsolvable})$ has a global solution, and this remains an open question. More generally, it is not known whether replacing $G_j$ in condition $(V_{i,j}2)$ with $F_i$ will result in a theorem similar to Theorem 3.3.
\end{example3}



\end{document}